\documentclass[12pt,oneside,english]{amsart}
\usepackage[T1]{fontenc}
\usepackage[latin9]{inputenc}
\usepackage[a4paper]{geometry}
\usepackage{babel}
\usepackage{prettyref}
\usepackage{mathrsfs}
\usepackage{mathtools}
\usepackage{amstext}
\usepackage{amsthm}
\usepackage{amssymb}
\usepackage[unicode=true,
 bookmarks=false,
 breaklinks=false,pdfborder={0 0 1},backref=section,colorlinks=false]
 {hyperref}

\makeatletter
\numberwithin{equation}{section}
\numberwithin{figure}{section}
\theoremstyle{plain}
\newtheorem{thm}{\protect\theoremname}
\theoremstyle{plain}
\newtheorem{prop}[thm]{\protect\propositionname}
\theoremstyle{plain}
\newtheorem{cor}[thm]{\protect\corollaryname}
\theoremstyle{remark}
\newtheorem{rem}[thm]{\protect\remarkname}
\theoremstyle{plain}
\newtheorem{lem}[thm]{\protect\lemmaname}
\theoremstyle{definition}
\newtheorem{example}[thm]{\protect\examplename}

\usepackage{babel}
\makeatletter
\@namedef{subjclassname@2020}{%
  \textup{2020} Mathematics Subject Classification}
\makeatother

\makeatother

\providecommand{\corollaryname}{Corollary}
\providecommand{\examplename}{Example}
\providecommand{\lemmaname}{Lemma}
\providecommand{\propositionname}{Proposition}
\providecommand{\remarkname}{Remark}
\providecommand{\theoremname}{Theorem}

\begin{document}
\title[Semi-positive Quantization]{Geometric quantization results
for semi-positive line bundles on
a Riemann surface}

\author{George Marinescu}
\address{Universit{\"a}t zu K{\"o}ln, Mathematisches Institut, 
Weyertal 86-90, 50931 K{\"o}ln, 
Deutschland    \newline
\mbox{\quad}\,Institute of Mathematics `Simion Stoilow', 
Romanian Academy, Bucharest, Romania}
\email{gmarines@math.uni-koeln.de}
\author{Nikhil Savale}
\thanks{G.\ M.\ and N.\ S.\ are partially supported by the DFG funded project
CRC/TRR 191 `Symplectic Structures in Geometry, Algebra and Dynamics'
(Project-ID 281071066-TRR 191). G.\ M.\ 
is partially supported by the DFG Priority Program 2265 
`Random Geometric Systems' (Project-ID 422743078)
and the ANR-DFG project QuaSiDy (Project-ID 490843120).}
\address{Universität zu Köln, Mathematisches Institut, Weyertal 86-90, 50931
Köln, Germany}
\email{nsavale@math.uni-koeln.de}
\begin{abstract}
In earlier work \cite{Marinescu-Savale18} the authors proved the
Bergman kernel expansion for semipositive line bundles over a Riemann
surface whose curvature vanishes to atmost finite order at each point.
Here we explore the related results and consequences of the expansion
in the semipositive case including: Tian's approximation theorem for
induced Fubini-Study metrics, leading order asymptotics and composition
for Toeplitz operators, asymptotics of zeroes for random sections
and the asymptotics of holomorphic torsion. 
\end{abstract}

\maketitle
\tableofcontents{}

\section{Introduction}

Geometric quantization is a procedure to relate classical observables
(smooth functions) on a phase space (a symplectic manifold) to quantum
observables (bounded linear operators) on the corresponding quantum
space (sections of a line bundle). In the case when the line bundle
in question is positive, and consequently the underlying manifold
Kahler, a well known quantization recipe is that of Berezin-Toeplitz
\cite{Berezin-Quantization-74,Ma-ICM-2010,Schlichenmaier-survey-2010}.
Showing the validity of the quantization prodecure involves proving
that it has the right properties in the semiclassical limit. Key to
the proof is the analysis of the semiclassical limit of the Bergman
kernel \cite{Catlin97-Bergmankernel,Dai-Liu-Ma2006-Bergman-kernel,MaMarinescu2008,Ma-Marinescu,Zelditch98-Bergmankernel}.
In earlier work \cite{Marinescu-Savale18} the authors proved the
Bergman kernel expansion in the case when the underlying line bundle
is only semipositive, with curvature vanishing at finite order at
each point, on a Riemann surface. It is the purpose of this article
to explore the corresponding applications of the expansion therein
to results in geometric quantization in the semi-positive case. These
include Tian's approximation theorem for induced Fubini-Study metrics,
leading order asymptotics and composition for Toeplitz operators,
asymptotics of zeroes for random sections and the asymptotics of holomorphic
torsion. 

We now state our results more precisely. Let $Y^{2}$ be a compact
Riemannian surface with complex structure $J$ and Hermitian metric
$h^{TY}$. Consider holomorphic, Hermitian line and vector bundles
$\left(L,h^{L}\right)$, $\left(F,h^{F}\right)$ on $Y$ and let $\nabla^{L},\nabla^{F}$
be the corresponding Chern connections. Denote by $R^{L}=\left(\nabla^{L}\right)^{2}\in\Omega^{2}\left(Y;i\mathbb{R}\right)$
the corresponding curvature of the line bundle. The order of vanishing
of $R^{L}$ at a point $y\in Y$ is now defined 
\begin{equation}
r_{y}-2=\textrm{ord}_{y}\left(R^{L}\right)\coloneqq\min\left\{ l|J^{l}\left(\Lambda^{2}T^{*}Y\right)\ni j_{y}^{l}R^{L}\neq0\right\} ,\quad r_{y}\geq2,\label{eq:order vanishing curvature}
\end{equation}
where $j^{l}R^{L}$ denotes the $l$th jet of the curvature. We shall
assume that this order of vanishing is finite at any point of the
manifold i.e. 
\begin{equation}
r\coloneqq\max_{y\in Y}r_{y}<\infty.\label{eq: curv vanishes finite order}
\end{equation}
The function $y\mapsto r_{y}$ being upper semi-continuous then gives
a decomposition of the manifold $Y=\bigcup_{j=2}^{r}Y_{j}$; $Y_{j}\coloneqq\left\{ y\in Y|r_{y}=j\right\} $
with each $Y_{\leq j}\coloneqq\bigcup_{j'=0}^{j}Y_{j'}$ being open.
Furthermore, the curvature is assumed to be semipositive: $R^{L}\left(w,\bar{w}\right)\geq0$,
for all $w\in T^{1,0}Y$.

Associated to the above one has the Kodaira Laplacian 
\[
\Box_{k}^{q}:\Omega^{0,q}\left(Y;F\otimes L^{k}\right)\rightarrow\Omega^{0,q}\left(Y;F\otimes L^{k}\right),\quad0\leq q\leq1,
\]
acting on tensor powers. The kernel of the Kodaira Laplacian $\ker\Box_{k}^{q}=H^{q}\left(X;F\otimes L^{k}\right)$
is cohomological and corresponds to holomorphic sections. The Bergman
kernel $\Pi_{k}^{q}\left(y,y'\right)$ is the Schwartz kernel of the
orthogonal projector $\Pi_{k}^{q}:\Omega^{0,q}\left(Y;F\otimes L^{k}\right)\rightarrow\ker\Box_{k}^{q}$.
Its value on the diagonal is 
\[
\Pi_{k}^{q}\left(y,y\right)=\sum_{j=1}^{N_{k}^{q}}\left|s_{j}\left(y\right)\right|^{2},\quad N_{k}^{q}\coloneqq\dim H^{q}\left(X;F\otimes L^{k}\right),
\]
for an orthonormal basis $\left\{ s_{j}\right\} _{j=1}^{N_{k}^{q}}$
of $H^{q}\left(X;F\otimes L^{k}\right)$. Under these assumptions
one has $H^{1}\left(X;F\otimes L^{k}\right)=0$ for $k\gg0$. We now
first recall our theorem from \cite{Marinescu-Savale18} on the asymptotics
of the Bergman kernel $\Pi_{k}\coloneqq\Pi_{k}^{0}$. 
\begin{thm}[{\cite[Theorem 3]{Marinescu-Savale18}}]
\label{thm:Bergman kernel expansion}
Let $Y$ be a compact Riemann surface and $(L,h^{L})\to Y$ a semipositive
line bundle whose curvature $R^{L}$ vanishes to finite order at any
point. Let $(F,h^{F})\to Y$ be another Hermitian holomorphic vector
bundle. Then the Bergman kernel $\Pi_{k}\coloneqq\Pi_{k}^{0}$ has
the pointwise asymptotic expansion on diagonal 
\begin{equation}
\Pi_{k}\left(y,y\right)=k^{2/r_{y}}\left[\sum_{j=0}^{N}c_{j}\left(y\right)k^{-2j/r_{y}}\right]+O\left(k^{-2N/r_{y}}\right),\quad\forall N\in\mathbb{N}.\label{eq:Bergmankernelexpansion}
\end{equation}
Here $c_{j}$ are sections of $\textrm{End}\left(F\right)$, with
the leading term $c_{0}\left(y\right)=\Pi^{g_{y}^{TY},j_{y}^{r_{y}-2}R^{L},J_{y}^{TY}}\left(0,0\right)>0$
being given in terms of the Bergman kernel of the model Kodaira Laplacian
on the tangent space at $y$ \prettyref{eq: model Kodaira Laplace}. 
\end{thm}

To explain our first consequence of the above note that the cohomology
$H^{0}(Y;F\otimes L^{k})$ is endowed with an $L^{2}$ product induced
by $h^{TY}$, $h^{L}$ and $h^{F}$. This induces a Fubini-Study metric
$\omega_{FS}$ on the projective space $\mathbb{P}\left[H^{0}\left(Y;F\otimes L^{k}\right)^{*}\right]$.
The Kodaira map is now defined
\begin{align}
\Phi_{k}:Y & \rightarrow\mathbb{P}\left[H^{0}\left(Y;F\otimes L^{k}\right)^{*}\right],\nonumber \\
\Phi_{k}\left(y\right) & \coloneqq\left\{ s\in H^{0}\left(Y;F\otimes L^{k}\right)|s\left(y\right)=0\right\} .\label{eq:Kodaira map-1}
\end{align}
It is well known that the map is holomorphic. We now have the semi-positive
version of Tian's approximation theorem.
\begin{thm}
\label{thm:Tian's approx theorem} Let $Y$ be a compact Riemann surface
and $(L,h^{L})$, $(F,h^{F})$ be holomorphic Hermitian line bundles
on $Y$ such that $(L,h^{L})$ is semi-positive and its curvature
vanishes at most at finite order. Then the Fubini-Study forms induced
by the Kodaira map \prettyref{eq:Kodaira map-1} converge uniformly
on $Y$ to the curvature $R^{L}$ of the line bundle with speed $k^{-1/3}$
\[
\left\Vert \frac{1}{k}\Phi_{k}^{*}\omega_{FS}-\frac{i}{2\pi}R^{L}\right\Vert _{C^{0}\left(Y\right)}=O\left(k^{-1/3}\right)
\]
as $k\to\infty$. 
\end{thm}

For the next application we consider the Toeplitz quantization of
functions on $Y$, or more generally sections of $F$. The Toeplitz
operator $T_{f,k}$ operator corresponding to a section $f\in C^{\infty}\left(Y;\textrm{End}\left(F\right)\right)$
is defined via 
\begin{align}
T_{f,k} & :C^{\infty}\left(Y;F\otimes L^{k}\right)\rightarrow C^{\infty}\left(Y;F\otimes L^{k}\right)\nonumber \\
T_{f,k} & \coloneqq\Pi_{k}f\Pi_{k},\label{eq:standard Toeplitz-1}
\end{align}
where $f$ denotes the operator of pointwise composition by $f$.
Each Toeplitz operator above further maps $H^{0}\left(Y;F\otimes L^{k}\right)$
to itself. A generalized Toeplitz operator, see \ref{gen Toeplitz operator}
below, acting on $H^{0}\left(Y;F\otimes L^{k}\right)$ is defined
as one having an asymptotic expansion in $k^{-1}$ with coefficients
being the Toeplitz operators \prettyref{eq:standard Toeplitz-1} as
above. Our next result is now as follows.
\begin{thm}
\label{thm:Toeplitz quantization} Let $(L,h^{L})$ and $(F,h^{F})$
be Hermitian holomorphic line bundles on a compact Riemann surface
$Y$ and assume that $(L,h^{L})$ is semi-positive line bundle and
its curvature $R^{L}$ vanishes to finite order at any point. Given
$f,g\in C^{\infty}\left(Y;\textrm{End}\left(F\right)\right)$, the
Toeplitz operators \prettyref{eq:standard Toeplitz} satisfy 
\begin{align}
\lim_{k\rightarrow\infty}\left\Vert T_{f,k}\right\Vert  & =\left\Vert f\right\Vert _{\infty}\coloneqq\sup_{\substack{y\in Y\\
u\in F_{y}\setminus0
}
}\frac{|f(y)u|_{h^{F}}}{|u|_{h^{F}}}\,,\label{eq: norm Toeplitz-1}\\
T_{f,k}T_{g,k} & =T_{fg,k}+O_{L^{2}\rightarrow L^{2}}\left(k^{-1/r}\right).\label{eq: leading comp. Toeplit-1}
\end{align}
Moreover, the space of generalized Toeplitz operators supported on
the subset $Y_{2}$ where the curvature is positive form an algebra
under operator addition and composition. 
\end{thm}

For our next result, we consider the asymptotics of zeroes of random
sections associated to tensor powers. To state the result first note
that the natural $L^{2}$ metric on $H^{0}\left(Y;F\otimes L^{k}\right)$
gives rise to a probability density $\mu_{k}$ on the sphere 
\[
SH^{0}\left(Y;F\otimes L^{k}\right)\coloneqq\left\{ s\in H^{0}\left(Y;F\otimes L^{k}\right)|\left\Vert s\right\Vert =1\right\} ,
\]
of finite dimension $\chi\left(Y;F\otimes L^{k}\right)-1$ \prettyref{eq: hol. Euler characteristic}.
We now define the product probability space $\left(\Omega,\mu\right)\coloneqq\left(\Pi_{k=1}^{\infty}SH^{0}\left(Y;F\otimes L^{k}\right),\Pi_{k=1}^{\infty}\mu_{k}\right)$.
To a random sequence of sections $s=\left(s_{k}\right)_{k\in\mathbb{N}}\in\Omega$
given by this probability density, we then associate the random sequence
of zero divisors $Z_{s_{k}}=\left\{ s_{k}=0\right\} $ and view it
as a random sequence of currents of integration in $\Omega_{0,0}\left(Y\right)$.
We now have the following. 
\begin{thm}
\label{thm: random section-1} Let $(L,h^{L})$ and $(F,h^{F})$ be
Hermitian holomorphic line bundles on a compact Riemann surface $Y$
and assume that $(L,h^{L})$ is semi-positive line bundle and its
curvature $R^{L}$ vanishes to finite order at any point. Then for
$\mu$-almost all $s=\left(s_{k}\right)_{k\in\mathbb{N}}\in\Omega$,
the sequence of currents 
\[
\frac{1}{k}Z_{s_{k}}\rightharpoonup\frac{i}{2\pi}R^{L}
\]
converges weakly to the semi-positive curvature form. 
\end{thm}

Our final result concerns the asymptotics of holomorphic torsion. 
\begin{thm}
\label{thm:holomorphic torsion} Let $(L,h^{L})$ and $(F,h^{F})$
be Hermitian holomorphic line bundles on a compact Riemann surface
$Y$ and assume that $(L,h^{L})$ is semi-positive line bundle and
its curvature $R^{L}$ vanishes to finite order at any point. The
holomorphic torsion satisfies the asymptotics 
\[
\ln\mathcal{T}_{k}\coloneqq-\frac{1}{2}\zeta'_{k}\left(0\right)=-k\ln k\left[\frac{\tau^{L}}{8\pi}\right]-k\left[\frac{\tau^{L}}{8\pi}\ln\left(\frac{\tau^{L}}{2\pi}\right)\right]+o\left(k\right)
\]
as $k\rightarrow\infty$. 
\end{thm}

All of our results above are well known in the case when the line
bundle $L$ is positive. In the positive case, the leading term of
the Bergman kernel expansion \prettyref{thm:Bergman kernel expansion}
was first shown in \cite{Tian90} and thereafter improved to a full
expansion in \cite{Catlin97-Bergmankernel,Zelditch98-Bergmankernel}
as a consequence of the Boutet de Monvel-Sjöstrand parametrix \cite{Boutet-Sjostrand76}
for the Szeg\H{o} kernel of a strongly pseudoconvex CR manifold. Subsequently
a different geometric method for the expansion was developed in \cite{Dai-Liu-Ma2006-Bergman-kernel,Ma-Marinescu}
inspired by the analytic localization method of \cite{Bismut-Lebeau91}.
The application of the Bergman kernel to induced Fubini-Study metrics
\prettyref{thm:Tian's approx theorem} is also found in \cite{Tian90}
in the positive case. The construction of the full Toeplitz algebra,
along with the properties of Toeplitz operators, was first done in
\cite{Bordemann-Meinrenken-Schlichenmaier94} as an application of
the the Boutet de Monvel-Guillemin calculus of Toeplitz operators
\cite{Boutet-Guillemin81}. The equidistribution result for random
sections in the positive case was first done in \cite{Shiffman-Zelditch99},
and \cite{DS06} gave also the speed of convergence of the zero-divisors.
Finally, the asymptotics of holomorphic torsion for positive line
bundles is due to Bismut-Vasserot \cite{Bismut-Vasserot}.

In the semi-positive case our results are new. The Bergman kernel
expansion \prettyref{thm:Bergman kernel expansion} was shown by the
authors in their earlier work \cite{Marinescu-Savale18}. The corresponding
problem for the Szeg\H{o} kernel of a weakly pseudoconvex CR manifold
in dimension three was solved by the second author in \cite{HsiaoSavale-2022}.
The expansion proved in \cite[Theorem 3]{Marinescu-Savale18} is however
only pointwise along the diagonal. In order to obtain the approximation
for Fubini-Study metrics \prettyref{thm:Tian's approx theorem} one
needs to prove uniform estimates on the Bergman kernel and its derivatives.
The composition for Toeplitz operators supported on the subset where
the curvature is positive in \prettyref{thm:Toeplitz quantization}
was shown earlier by the first author in \cite[Theorem 1.4]{Hsiao-Marinescu2017}
under the assumption of a small spectral gap for the Kodaira Laplacian.
A more general result, than the equidistribution for zeroes of a random
holomorphic section of a semipositive line bundle, was obtained in
\cite[S 4]{Dinh-Ma-Marinescu2016} using $L^{2}$ estimates for the
$\bar{\partial}$-equation of a modified positive metric.

The paper is organized as follows. In \prettyref{sec:sub-Riemannian-geometry}
we begin with some standard preliminaries. These include the relevant
spectral gap properties for the Bochner and Kodaira Laplacians in
subsections \prettyref{subsec:sR-and-Bochner} and \prettyref{subsec:Kodaira-Laplacian-and}
respectively. In\prettyref{subsec:Bergman-kernel exp.} we recall
the proof of the pointwise Bergman kernel expansion from \cite{Marinescu-Savale18}.
In \prettyref{subsec:Uniform-estimates-on} we further derive uniform
estimates on semipositive Bergman kernels that are necessary for the
applications in this article. In \prettyref{subsection:Tian} we use
the uniform Bergman kernel estimates to prove the semipositive version
of Tian's theorem \prettyref{thm:Tian's approx theorem}. In \prettyref{sec:Toeplitz-operators}
we prove the analogous expansion for the kernel of a Toeplitz operator
and the corresponding theorem \prettyref{thm:Toeplitz quantization}
on Toeplitz quantization. In \prettyref{subsec:Random-sections} we
prove the equidistribution result \prettyref{thm: random section-1}
for random sections. In the final \prettyref{sec:Holomorphic-torsion}
we prove the asymptotic result for holomorphic torsion \prettyref{thm:holomorphic torsion}.
The final appendix \prettyref{sec:Model-operators} describes facts
on model Laplacians and Bergman kernels that are used throughout the
article.

\section{\label{sec:sub-Riemannian-geometry} Preliminaries}

Here we begin with some preliminary notions. Let $Y$ be a compact
Riemann surface. It is equipped with an integrable complex structure
$J$ and Hermitian metric $h^{TY}$ on its complex tangent space.
Also denote by $g^{TY}$ the associated Riemannian metric on $TY$.
Next let $(L,h^{L})$, $(F,h^{F})$ be an auxiliary pair of Hermitian,
holomorphic bundles where $L$ is of rank one. We denote by $\nabla^{L}$,
$\nabla^{F}$ the corresponding Chern connections and $R^{L}$, $R^{F}$
their corresponding curvatures. The order of vanishing $r_{y}$ of
the curvature $R^{L}$ at a point $y\in Y$ is now defined as in \prettyref{eq:order vanishing curvature}.
And we assume that the curvature $R^{L}$ vanishes at finite order
at any point of $Y$ i.e. 
\begin{equation}
r\coloneqq\max_{y\in Y}r_{y}<\infty.\label{eq: curv vanishes finite order-1}
\end{equation}
The curvature $R^{L}$ of $\nabla^{L}$ is a $\left(1,1\right)$ form
which is further assumed to be semi-positive 
\begin{align}
iR^{L}\left(v,Jv\right) & \geq0,\quad\forall v\in TY\quad\textrm{ or equivalently}\nonumber \\
R^{L}\left(w,\bar{w}\right) & \geq0,\quad\forall w\in T^{1,0}Y.\label{eq:semi-positivity-1}
\end{align}
We note that semipositivity implies that the order of vanishing $r_{y}-2\in2\mathbb{N}_{0}$
of the curvature $R^{L}$ at any point $y$ is even. Semipositivity
and finite order of vanishing imply that there are points where the
curvature is positive (the set where the curvature is positive is
in fact an open dense set). Hence 
\[
\deg L=\int_{Y}c_{1}(L)=\int_{Y}\frac{i}{2\pi}R^{L}>0,
\]
so that $L$ is ample.

\subsection{\label{subsec:sR-and-Bochner} sR and Bochner Laplacians}

Associated to the above data one has the Bochner Laplacian on tensor
powers defined by 
\begin{equation}
\Delta_{k}\coloneqq\left(\nabla^{F\otimes L^{k}}\right)^{*}\nabla^{F\otimes L^{k}}:C^{\infty}\left(Y;F\otimes L^{k}\right)\rightarrow C^{\infty}\left(Y;F\otimes L^{k}\right),\label{eq:Bochner Laplacian}
\end{equation}
for each $k\in\mathbb{N}$, with the adjoint above being taken with
respect to the corresponding metrics and the Riemannian volume form. 

Each Bochner Laplacian \prettyref{eq:Bochner Laplacian} above is
the Fourier mode of a sub-Riemannian (sR) Laplacian on the unit circle
bundle of $L$. To elaborate, denote by $X=S^{1}L\rightarrow Y$ the
unit circle bundle of the line bundle $L$. Further let $E\coloneqq HX\subset TX$
be the horizontal distribution. The distribution carries the metric
$g^{E}=\pi^{*}g^{TY}$ pulled back from the base. We also denote by
the same notation the pullback of $\left(F,h^{F},\nabla^{F}\right)$
from $Y$ to $X$. The finite order of vanishing for the curvature
$R^{L}$ in \prettyref{eq: curv vanishes finite order} is equivalent
to the \textit{bracket generating }condition for the distribution
$E$: the Lie brackets in $C^{\infty}\left(E\right)$ generates all
vector fields $C^{\infty}\left(TX\right)$ \cite[Prop. 6]{Marinescu-Savale18}.
As such the triple $\left(X,E\subset TX,g^{E}\right)$ is a \textit{sub-Riemannian
(sR) manifold}. Furthermore the maximum order of vanishing for the
curvature $r$ \prettyref{eq: curv vanishes finite order} is then
the degree of non-holonomy of the distribution $E$, i.e. the number
of brackets required to generate the missing vertical direction. A
volume form on $X$ is defined via $\mu_{X}\coloneqq\mu_{g^{TY}}\wedge e^{*}$
with $\mu_{g^{TY}}$ denoting the Riemannian volume form on $Y$ and
$e^{*}$ being the dual one form to the generating $e\in C^{\infty}\left(TX\right)$
of the circle action on $X$. 

The subRiemannian Laplacian on $X$ 
\begin{align}
\Delta_{g^{E},\mu_{X}} & :C^{\infty}\left(X;F\right)\rightarrow C^{\infty}\left(X;F\right)\nonumber \\
\Delta_{g^{E},\mu_{X}} & \coloneqq\left(\nabla^{g^{E},F}\right)_{\mu_{X}}^{*}\circ\nabla^{g^{E},F}\label{eq:sR Laplacian}
\end{align}
being the composition of the sR gradient defined via
\begin{align}
\nabla^{g^{E},F} & :C^{\infty}\left(X;E\right)\rightarrow C^{\infty}\left(X;E\otimes F\right),\nonumber \\
h^{E,F}\left(\nabla^{g^{E},F}s,v\otimes s'\right) & \coloneqq h^{F}\left(\nabla_{v}^{F}s,s'\right),\label{eq:sR gradient}
\end{align}
$\forall v\in C^{\infty}\left(X;E\right),s'\in C^{\infty}\left(X;F\right),$
where $h^{E,F}\coloneqq g^{E}\otimes h^{F}$, with its adjoint taken
with respect to $\mu_{X}$. Under the bracket generating condition,
the sR Laplacian satisfies the sharp subelliptic estimate of Rothschild
and Stein with a gain of $\frac{1}{r}$ derivatives 
\begin{equation}
\left\Vert \psi s\right\Vert _{H^{1/r}}^{2}\leq C\left[\left\langle \Delta_{g^{E},F,\mu}\varphi s,\varphi s\right\rangle +\left\Vert \varphi s\right\Vert _{L^{2}}^{2}\right],\quad\forall s\in C^{\infty}\left(X;F\right)\label{eq:local subelliptic estimate}
\end{equation}
for all $\varphi,\psi\in C_{c}^{\infty}\left(X\right)$, with $\varphi=1$
on $\textrm{spt}\left(\psi\right)$, and where $r$ is again given
by \prettyref{eq: curv vanishes finite order} and corresponds to
the maximum step size of the distribution $E$. 

Next, the unit circle bundle of $L$ being $X$, the pullback $\mathbb{C}\cong\pi^{*}L\rightarrow X$
is canonically trivial via the identification $\pi^{*}L\ni\left(x,l\right)\mapsto x^{-1}l\in\mathbb{C}$.
Pulling back sections then gives the identification 
\begin{equation}
C^{\infty}\left(X;F\right)=\oplus_{k\in\mathbb{Z}}C^{\infty}\left(Y;F\otimes L^{k}\right).\label{eq:Fourier decomposition}
\end{equation}
Each summand on the right hand side above corresponds to an eigenspace
of $\nabla_{e}^{F}$ with eigenvalue $-ik$. While horizontal differentiation
$d^{H}$ on the left corresponds to differentiation with respect to
the tensor product connection$\nabla^{L^{k}}$ on the right hand side
above. Pick an invariant density $\mu_{X}$ on $X$ inducing a density
$\mu_{Y}$ on $Y$. This now defines the sR Laplacian $\Delta_{g^{E},F,\mu_{X}}$
acting on sections of $F$. By invariance the sR Laplacian commutes
$\left[\Delta_{g^{E},F,\mu_{X}},e\right]=0$ with the generator of
the circle action and hence preserves the decomposition \prettyref{eq:Fourier decomposition}.
It acts via 
\begin{equation}
\Delta_{g^{E},F,\mu_{X}}=\oplus_{k\in\mathbb{Z}}\Delta_{k}\label{eq: Fourier decomposition Laplacian}
\end{equation}
on each component where $\Delta_{k}$ is the Bochner Laplacian \prettyref{eq:Bochner Laplacian}
on the tensor powers $F\otimes L^{k}$, with adjoint being taken with
respect to $\mu_{g^{TY}}$. 

Using the description of the Bochner Laplacian as the Fourier mode
of the sR Laplacian \prettyref{eq: Fourier decomposition Laplacian},
in \cite[Thm. 1]{Marinescu-Savale18} a general leading asymptotic
result for the first positive eigenvalues was proved. Here we recall
a simple argument for its lower bound.
\begin{prop}
\label{prop: gen. spectral gap} There exist constants $c_{1},c_{2}>0$,
such that one has $\textrm{Spec}\left(\Delta_{k}\right)\subset\left[c_{1}k^{2/r}-c_{2},\infty\right)$
for each $k$. 
\end{prop}

\begin{proof}
The subelliptic estimate \prettyref{eq:local subelliptic estimate}
on the circle bundle is 
\[
\left\Vert \partial_{\theta}^{1/r}s\right\Vert ^{2}\leq\left\Vert s\right\Vert _{H^{1/r}}^{2}\leq C\left[\left\langle \Delta_{g^{E},F,\mu_{X}}s,s\right\rangle +\left\Vert s\right\Vert _{L^{2}}^{2}\right],\;\forall s\in C^{\infty}\left(X;F\right).
\]
Letting $s=\pi^{*}s'$ be the pullback of an orthonormal eigenfunction
$s'$ of $\Delta_{k}$ with eigenvalue $\lambda$ on the base gives
$k^{2/r}\leq C\left(\lambda+1\right)$as required. 
\end{proof}

\subsection{\label{subsec:Kodaira-Laplacian-and} Kodaira Laplacian and its spectral
gap}

Related to the Bochner Laplacian \prettyref{eq:Bochner Laplacian}
is the Kodaira Laplacian on tensor powers. Namely, with $\left(\Omega^{0,*}\left(X;F\otimes L^{k}\right);\bar{\partial}_{k}\right)$
denoting the Dolbeault complex the Kodaira Laplace and Dirac operators
acting on $\Omega^{0,*}\left(X;F\otimes L^{k}\right)$ are defined
\begin{align}
\Box_{k}\coloneqq\frac{1}{2}\left(D_{k}\right)^{2} & =\bar{\partial}_{k}\bar{\partial}_{k}^{*}+\bar{\partial}_{k}^{*}\bar{\partial}_{k}\label{eq: Kodaira Laplace}\\
D_{k} & \coloneqq\sqrt{2}\left(\bar{\partial}_{k}+\bar{\partial}_{k}^{*}\right).\label{eq: Kodaira Dirac}
\end{align}

Clearly, $D_{k}$ interchanges while $\Box_{k}$ preserves $\Omega^{0,0/1}$.
We denote $D_{k}^{\pm}=\left.D_{k}\right|_{\Omega^{0,0/1}}$ and $\Box_{k}^{0/1}=\left.\Box_{k}\right|_{\Omega^{0,0/1}}$.
The Clifford multiplication endomorphism $c:TY\rightarrow\textrm{End}\left(\Lambda^{0,*}\right)$
is defined via $c\left(v\right)\coloneqq\sqrt{2}\left(v^{1,0}\wedge-i_{v^{0,1}}\right)$,
$v\in TY$, and extended to the entire exterior algebra $\Lambda^{*}TY$
via $c\left(1\right)=1,\,c\left(v_{1}\wedge v_{2}\right)\coloneqq c\left(v_{1}\right)c\left(v_{2}\right)$,
$v_{1},v_{2}\in TY$.

Denote by $\nabla^{TY},\nabla^{T^{1,0}Y}$ the Levi-Civita and Chern
connections on the real and holomorphic tangent spaces as well as
by $\nabla^{T^{0,1}Y}$ the induced connection on the anti-holomorphic
tangent space. Denote by $\Theta$ the real $\left(1,1\right)$ form
defined by contraction of the complex structure with the metric $\Theta\left(.,.\right)=g^{TY}\left(J.,.\right)$.
This is clearly closed $d\Theta=0$ (or $Y$ is Kähler) and the complex
structure is parallel $\nabla^{TY}J=0$ or $\nabla^{TY}=\nabla^{T^{1,0}Y}\oplus\nabla^{T^{1,0}Y}$.

With the induced tensor product connection on $\Lambda^{0,*}\otimes F\otimes L^{k}$
being denoted via $\nabla^{\Lambda^{0,*}\otimes F\otimes L^{k}}$,
the Kodaira Dirac operator \prettyref{eq: Kodaira Dirac} is now given
by the formula 
\[
D_{k}=c\circ\nabla^{\Lambda^{0,*}\otimes F\otimes L^{k}}.
\]

Next we denote by $R^{F}$ the curvature of $\nabla^{F}$ and by $\kappa$
the scalar curvature of $g^{TY}$. Define the following endomorphisms
of $\Lambda^{0,*}$ 
\begin{align}
\omega\left(R^{F}\right) & \coloneqq R^{F}\left(w,\bar{w}\right)\bar{w}i_{\bar{w}}\nonumber \\
\omega\left(R^{L}\right) & \coloneqq R^{L}\left(w,\bar{w}\right)\bar{w}i_{\bar{w}}\nonumber \\
\omega\left(\kappa\right) & \coloneqq\kappa\bar{w}i_{\bar{w}}\nonumber \\
\tau^{F} & \coloneqq R^{F}\left(w,\bar{w}\right)\nonumber \\
\tau^{L} & \coloneqq R^{L}\left(w,\bar{w}\right)\label{eq:formulas Clifford-1}
\end{align}
in terms of an orthonormal section $w$ of $T^{1,0}Y$. The Lichnerowicz
formula for the above Dirac operator (\cite{Ma-Marinescu} Thm 1.4.7)
simplifies for a Riemann surface and is given by 
\begin{align}
2\Box_{k} & =D_{k}^{2}=\left(\nabla^{\Lambda^{0,*}\otimes F\otimes L^{k}}\right)^{*}\nabla^{\Lambda^{0,*}\otimes F\otimes L^{k}}+k\left[2\omega\left(R^{L}\right)-\tau^{L}\right]+\left[2\omega\left(R^{F}\right)-\tau^{F}\right]+\frac{1}{2}\omega\left(\kappa\right).\label{eq:Lichnerowicz for.}
\end{align}

We now have the following. 
\begin{prop}
\label{prop:Dk spectral estimate} Let $Y$ be a compact Riemann surface,
$(L,h^{L})\to Y$ a semi-positive line bundle whose curvature $R^{L}$
vanishes to finite order at any point. Let $(F,h^{F})\to Y$ be a
Hermitian holomorphic vector bundle. Then there exist constants $c_{1},c_{2}>0$,
such that 
\[
\left\Vert D_{k}s\right\Vert ^{2}\geq\left(c_{1}k^{2/r}-c_{2}\right)\left\Vert s\right\Vert ^{2}
\]
for all $s\in\Omega^{0,1}\left(Y;F\otimes L^{k}\right)$. 
\end{prop}

\begin{proof}
Writing $s=\left|s\right|\bar{w}\in\Omega^{0,1}\left(Y;F\otimes L^{k}\right)$
in terms of a local orthonormal section $\bar{w}$ gives 
\begin{equation}
\left\langle \left[2\omega\left(R^{L}\right)-\tau^{L}\right]s,s\right\rangle =R^{L}\left(w,\bar{w}\right)\left|s\right|^{2}\geq0\label{eq: semipositivity curvature estimate}
\end{equation}
from \prettyref{eq:semi-positivity-1}, \prettyref{eq:formulas Clifford-1}.
This gives 
\begin{align*}
\left\Vert D_{k}s\right\Vert ^{2} & =\left\langle D_{k}^{2}s,s\right\rangle \\
 & =\left\langle \left[\left(\nabla^{\Lambda^{0,*}\otimes F\otimes L^{k}}\right)^{*}\nabla^{\Lambda^{0,*}\otimes F\otimes L^{k}}+k\left[2\omega\left(R^{L}\right)-\tau^{L}\right]\right.\right.\\
 & \qquad\left.\left.+\left[2\omega\left(R^{F}\right)-\tau^{F}\right]+\frac{1}{2}\omega\left(\kappa\right)\right]s,s\right\rangle \\
 & \geq\left\langle \left(\nabla^{\Lambda^{0,*}\otimes F\otimes L^{k}}\right)^{*}\nabla^{\Lambda^{0,*}\otimes F\otimes L^{k}}s,s\right\rangle -c_{0}\left\Vert s\right\Vert ^{2}\\
 & \geq\left(c_{1}k^{2/r}-c_{2}\right)\left\Vert s\right\Vert ^{2}
\end{align*}
from Proposition \prettyref{prop: gen. spectral gap}, \prettyref{eq:Lichnerowicz for.}
and \prettyref{eq: semipositivity curvature estimate}. 
\end{proof}
We now derive as a corollary a spectral gap property for Kodaira Dirac/Laplace
operators $D_{k}$, $\Box_{k}$ corresponding to Proposition \prettyref{prop: gen. spectral gap}. 
\begin{cor}
\label{cor: spectral gap Dirac} Under the hypotheses of Proposition
\prettyref{prop:Dk spectral estimate} there exist constants $c_{1},c_{2}>0$,
such that $\textrm{Spec}\left(\Box_{k}\right)\subset\left\{ 0\right\} \cup\left[c_{1}k^{2/r}-c_{2},\infty\right)$
for each $k$. Moreover, $\ker D_{k}^{-}=0$ and $H^{1}\left(Y;F\otimes L^{k}\right)=0$
for $k$ sufficiently large. 
\end{cor}

\begin{proof}
From Proposition \prettyref{prop:Dk spectral estimate}, it is clear
that 
\begin{equation}
\textrm{Spec}\left(\Box_{k}^{1}\right)\subset\left[c_{1}k^{2/r}-c_{2},\infty\right)\label{eq:spectral gap Kodaira}
\end{equation}
for some $c_{1},c_{2}>0$ giving the second part of the corollary.
Moreover, the eigenspaces of $\left.D_{k}^{2}\right|_{\Omega^{0,0/1}}$
with non-zero eigenvalue being isomorphic by Mckean-Singer, the first
part also follows. 
\end{proof}
Since $L$ is ample, we know also by the Kodaira-Serre vanishing theorem
that $H^{1}\left(Y;F\otimes L^{k}\right)$ vanishes for $k$ sufficiently
large. If $F$ is also a line bundle this follows from the well known
fact that for a line bundle $E$ on $Y$ we have $H^{1}\left(Y;E\right)=0$
whenever $\deg E>2g-2$. It is however interesting to have a direct
analytic proof. Of course, the vanishing theorem for a semi-positive
line bundle works only in dimension one, see Remark \prettyref{rem: good remark}
below.

The vanishing $H^{1}\left(Y;F\otimes L^{k}\right)=0$ for $k$ sufficiently
large gives 
\begin{align}
\textrm{dim }H^{0}\left(Y;F\otimes L^{k}\right) & =\chi\left(Y;F\otimes L^{k}\right)\nonumber \\
 & =\int_{Y}ch\left(F\otimes L^{k}\right)\textrm{Td}\left(Y\right)\nonumber \\
 & =k\left[\textrm{rk}\left(F\right)\int_{Y}c_{1}\left(L\right)\right]+\int_{Y}c_{1}\left(F\right)+1-g,\label{eq: hol. Euler characteristic}
\end{align}
by Riemann-Roch, with $\chi\left(Y;F\otimes L^{k}\right)$, $ch\left(F\otimes L^{k}\right)$,
$\textrm{Td}\left(Y\right)$, $g$ denoting the holomorphic Euler
characteristic, Chern character, Todd genus and genus of $Y$ respectively. 
\begin{rem}
\label{rem: good remark} The argument for Proposition \prettyref{prop:Dk spectral estimate}
breaks down in higher dimensions since there are more components to
$\left[2\omega\left(R^{L}\right)-\tau^{L}\right]$ in the Lichnerowicz
formula \prettyref{eq:Lichnerowicz for.} which semi-positivity is
insufficient to control. Indeed, there is a known counterexample to
the existence of a spectral gap for semi-positive line bundles in
higher dimensions due to Donnelly \cite{Donnelly2003}. 
\end{rem}

\section{\label{subsec:Bergman-kernel exp.} Bergman kernel expansion}

In this section we now first recall the expansion for the Bergman
kernel proved in \cite[Sec 4.1]{Marinescu-Savale18}. First recall
that the Bergman kernel is the Schwartz kernel $\Pi_{k}\left(y_{1},y_{2}\right)$
of the projector onto the nullspace of $\Box_{k}$ 
\begin{equation}
\Pi_{k}:C^{\infty}\left(Y;F\otimes L^{k}\right)\rightarrow\ker\left(\left.\Box_{k}\right|_{C^{\infty}\left(Y;F\otimes L^{k}\right)}\right),\label{eq: Bergman projector}
\end{equation}
with respect to the $L^{2}$ inner product given by the metrics $g^{TY}$,
$h^{F}$ and $h^{L}$. Alternately, if $s_{1},s_{2},\ldots,s_{N_{k}}$
denotes an orthonormal basis of eigensections of $H^{0}\left(X;F\otimes L^{k}\right)$
then 
\begin{equation}
\Pi_{k}\left(y_{1},y_{2}\right)=\sum_{j=1}^{N_{k}}s_{j}\left(y_{1}\right)\otimes s_{j}\left(y_{2}\right)^{*}.\label{eq: Bergman kernel}
\end{equation}
We wish to describe the asymptotics of $\Pi_{k}$ along the diagonal
in $Y\times Y$.

Consider $y\in Y$, and fix orthonormal bases $\left\{ e_{1},e_{2}\left(=Je_{1}\right)\right\} $,
$\left\{ l\right\} $, $\left\{ f_{j}\right\} _{j=1}^{\textrm{rk}\left(F\right)}$
for $T_{y}Y$, $L_{y}$ , $F$ respectively and let $\big\{ w\coloneqq\frac{1}{\sqrt{2}}\left(e_{1}-ie_{2}\right)\big\}$
be the corresponding orthonormal frame for $T_{y}^{1,0}Y$. Using
the exponential map from this basis obtain a geodesic coordinate system
on a geodesic ball $B_{2\varrho}\left(y\right)$. Further parallel
transport these bases along geodesic rays using the connections $\nabla^{T^{1,0}Y}$,
$\nabla^{L}$, $\nabla^{F}$ to obtain orthonormal frames for $T^{1,0}Y$,
$L$, $F$ on $B_{2\varrho}\left(y\right)$. In this frame and coordinate
system, the connection on the tensor product again has the expression
\begin{align}
\nabla^{\Lambda^{0,*}\otimes F\otimes L^{k}} & =d+a^{\Lambda^{0,*}}+a^{F}+ka^{L}\nonumber \\
a_{j}^{\Lambda^{0,*}} & =\int_{0}^{1}d\rho\left(\rho y^{k}R_{jk}^{\Lambda^{0,*}}\left(\rho x\right)\right)\nonumber \\
a_{j}^{F} & =\int_{0}^{1}d\rho\left(\rho y^{k}R_{jk}^{F}\left(\rho x\right)\right)\nonumber \\
a_{j}^{L} & =\int_{0}^{1}d\rho\left(\rho y^{k}R_{jk}^{L}\left(\rho x\right)\right)\label{eq:curvature formulas for connection}
\end{align}
in terms of the curvatures of the respective connections. We now define
a modified frame $\left\{ \tilde{e}_{1},\tilde{e}_{2}\right\} $ on
$\mathbb{R}^{2}$ which agrees with $\left\{ e_{1},e_{2}\right\} $
on $B_{\varrho}\left(y\right)$ and with $\left\{ \partial_{x_{1}},\partial_{x_{2}}\right\} $
outside $B_{2\varrho}\left(y\right)$. Also define the modified metric
$\tilde{g}^{TY}$ and almost complex structure $\tilde{J}$ on $\mathbb{R}^{2}$
to be standard in this frame and hence agreeing with $g^{TY}$, $J$
on $B_{\varrho}\left(y\right)$. The Christoffel symbol of the corresponding
modified induced connection on $\Lambda^{0,*}$now satisfies 
\[
\tilde{a}^{\Lambda^{0,*}}=0\quad\textrm{ outside }B_{2\varrho}\left(y\right).
\]

With $r_{y}-2\in2\mathbb{N}_{0}$ being the order of vanishing of
the curvature $R^{L}$ as before, we may Taylor expand the curvature
as
\begin{align}
R^{L} & =\underbrace{\sum_{\left|\alpha\right|=
r-2}R_{pq,\alpha}y^{\alpha}dy_{p}dy_{q}}_{=R_{0}^{L}}+
O\left(y^{r-1}\right)\quad\textrm{ with }\label{eq:curvature expansion}\\
iR_{0}^{L}\left(e_{1},e_{2}\right) & \geq0.\label{eq: first term semi-positive}
\end{align}
Further we may define the modified connections $\tilde{\nabla}^{F},$
$\tilde{\nabla}^{L}$ via 
\begin{align}
\tilde{\nabla}^{F} & =d+\chi\left(\frac{\left|y\right|}{2\varrho}\right)a^{F}\nonumber \\
\tilde{\nabla}^{L} & =d+\left[\underbrace{\int_{0}^{1}d\rho\,\rho y^{k}\left(\tilde{R}^{L}\right)_{jk}\left(\rho y\right)}_{=\tilde{a}_{j}^{L}}\right]dy_{j},\quad\textrm{ where}\nonumber \\
\tilde{R}^{L} & =\chi\left(\frac{\left|y\right|}{2\varrho}\right)R^{L}+\left[1-\chi\left(\frac{\left|y\right|}{2\varrho}\right)\right]R_{0}^{L}.\label{eq:modified connection-1}
\end{align}
as well as the corresponding tensor product connection $\tilde{\nabla}^{\Lambda^{0,*}\otimes F\otimes L^{k}}$
which agrees with $\nabla^{\Lambda^{0,*}\otimes F\otimes L^{k}}$
on $B_{\varrho}\left(y\right)$. Clearly the curvature of the modified
connection $\tilde{\nabla}^{L}$ is given by $\tilde{R}^{L}$\prettyref{eq:modified connection-1}
and is semi-positive by \prettyref{eq: first term semi-positive}.
Equation \prettyref{eq:modified connection-1} also gives $\tilde{R}^{L}=R_{0}^{L}+O\left(\varrho^{r_{y}-1}\right)$
and that the $\left(r_{y}-2\right)$-th derivative/jet of $\tilde{R}^{L}$
is non-vanishing at all points on $\mathbb{R}^{2}$ for 
\begin{equation}
0<\varrho<c\left|j^{r_{y}-2}R^{L}\left(y\right)\right|.\label{eq: curv =00003D000026 jet comp.}
\end{equation}
Here $c$ is a uniform constant depending on the $C^{r-2}$ norm of
$R^{L}$. We now define the modified Kodaira Dirac operator on $\mathbb{R}^{2}$
by the similar formula 
\begin{equation}
\tilde{D}_{k}=c\circ\tilde{\nabla}^{\Lambda^{0,*}\otimes F\otimes L^{k}},\label{eq: local Dirac}
\end{equation}
agreeing with $D_{k}$ on $B_{\varrho}\left(y\right).$ This has a
similar Lichnerowicz formula

\begin{align}
\tilde{D}_{k}^{2}=2\tilde{\Box}_{k}\coloneqq & \left(\tilde{\nabla}^{\Lambda^{0,*}\otimes F\otimes L^{k}}\right)^{*}\tilde{\nabla}^{\Lambda^{0,*}\otimes F\otimes L^{k}}+k\left[2\omega\left(\tilde{R}^{L}\right)-\tilde{\tau}^{L}\right]\label{eq:model Laplace}\\
\quad & +\left[2\omega\left(\tilde{R}^{F}\right)-\tilde{\tau}^{F}\right]+\frac{1}{2}\omega\left(\tilde{\kappa}\right)
\end{align}
the adjoint being taken with respect to the metric $\tilde{g}^{TY}$
and corresponding volume form. Also the endomorphisms $\tilde{R}^{F},\tilde{\tau}^{F},\tilde{\tau}^{L}$
and $\omega\left(\tilde{\kappa}\right)$ are the obvious modifications
of \prettyref{eq:formulas Clifford-1} defined using the curvatures
of $\tilde{\nabla}^{F},\tilde{\nabla}^{L}$ and $\tilde{g}^{TY}$
respectively. The above \prettyref{eq:model Laplace} again agrees
with $\Box_{k}$ on $B_{\varrho}\left(y\right)$ while the endomorphisms
$\tilde{R}^{F},\tilde{\tau}^{F},\omega\left(\tilde{\kappa}\right)$
all vanish outside $B_{\varrho}\left(y\right)$. Being semi-bounded
below \prettyref{eq:model Laplace} is essentially self-adjoint. A
similar argument as Corollary \prettyref{cor: spectral gap Dirac}
gives a spectral gap 
\begin{equation}
\textrm{Spec}\left(\tilde{\Box}_{k}\right)\subset\left\{ 0\right\} \cup\left[c_{1}k^{2/r_{y}}-c_{2},\infty\right).\label{eq: spectral gap}
\end{equation}
Thus for $k\gg0$, the resolvent $\left(\tilde{\Box}_{k}-z\right)^{-1}$
is well-defined in a neighborhood of the origin in the complex plane.
On account on the local elliptic estimate, the projector $\tilde{\Pi}_{k}$
from $L^{2}\left(\mathbb{R}^{2};\Lambda_{y}^{0,*}\otimes F_{y}\otimes L_{y}^{\otimes k}\right)$
onto $\ker\left(\tilde{\Box}_{k}\right)$ then has a smooth Schwartz
kernel with respect to the Riemannian volume of $\tilde{g}^{TY}$.

We are now ready to prove the Bergman kernel expansion \prettyref{thm:Bergman kernel expansion}. 
\begin{proof}[Proof of \prettyref{thm:Bergman kernel expansion}]
\label{bergman exp proof}First choose $\varphi\in\mathcal{S}\left(\mathbb{R}_{s}\right)$
even satisfying $\hat{\varphi}\in C_{c}\left(-\frac{\varrho}{2},\frac{\varrho}{2}\right)$
and $\varphi\left(0\right)=1$. For $c>0$, set $\varphi_{1}\left(s\right)=1_{\left[c,\infty\right)}\left(s\right)\varphi\left(s\right)$.
On account of the spectral gap Corollary \prettyref{cor: spectral gap Dirac},
and as $\varphi_{1}$ decays at infinity, we have 
\begin{align}
\varphi\left(D_{k}\right)-\Pi_{k} & =\varphi_{1}\left(D_{k}\right)\quad\textrm{ with }\nonumber \\
\left\Vert D_{k}^{a}\varphi_{1}\left(D_{k}\right)\right\Vert _{L^{2}\rightarrow L^{2}} & =O\left(k^{-\infty}\right)\label{eq: function approx. projector}
\end{align}
for $a\in\mathbb{N}$. Combining the above with semiclassical Sobolev
and elliptic estimates gives 
\begin{equation}
\left|\varphi\left(D_{k}\right)-\Pi_{k}\right|_{C^{l}\left(Y\times Y\right)}=O\left(k^{-\infty}\right),\label{eq:bergman vs schw.}
\end{equation}
$\forall l\in\mathbb{N}_{0}$. Next we may write $\varphi\left(D_{k}\right)=\frac{1}{2\pi}\int_{\mathbb{R}}e^{i\xi D_{k}}\hat{\varphi}\left(\xi\right)d\xi$
via Fourier inversion. Since $D_{k}=\tilde{D}_{k}$ on $B_{\varrho}\left(y\right)$
and $\hat{\varphi}\in C_{c}\left(-\frac{\varrho}{2},\frac{\varrho}{2}\right)$,
we may use a finite propagation argument to conclude 
\[
\varphi\left(D_{k}\right)\left(.,y\right)=\varphi\left(\tilde{D}_{k}\right)\left(.,0\right).
\]
By similar estimates as \prettyref{eq: function approx. projector}
for $\tilde{D}_{k}$ we now have a localization of the Bergman kernel
\begin{align}
\Pi_{k}\left(.,y\right) & =O\left(k^{-\infty}\right),\quad\textrm{ on }B_{\varrho}\left(y\right)^{c}\nonumber \\
\Pi_{k}\left(.,y\right)-\tilde{\Pi}_{k}\left(.,0\right) & =O\left(k^{-\infty}\right),\quad\textrm{ on }B_{\varrho}\left(y\right).\label{eq:Bergman localization}
\end{align}
It thus suffices to consider the Bergman kernel of the model Kodaira
Laplacian \prettyref{eq:model Laplace} on $\mathbb{R}^{2}$.

Next with the rescaling/dilation $\delta_{k^{-1/r}}y=\left(k^{-1/r}y_{1},\ldots,k^{-1/r}y_{n-1}\right)$,
the rescaled Kodaira Laplacian 
\begin{equation}
\boxdot\coloneqq k^{-2/r_{y}}\left(\delta_{k^{-1/r}}\right)_{*}\tilde{\Box}_{k}\label{eq:rescaled Dirac}
\end{equation}
satisfies 
\begin{align}
\varphi\left(\frac{\tilde{\Box}_{k}}{k^{2/r_{y}}}\right)\left(y,y'\right) & =k^{2/r_{y}}\varphi\left(\boxdot\right)\left(yk^{1/r_{y}},y'k^{1/r_{y}}\right)\label{eq: rescaling Schw kernel}
\end{align}
for $\varphi\in\mathcal{S}\left(\mathbb{R}\right)$. Using a Taylor
expansion via \prettyref{eq:modified connection-1}, \prettyref{eq: local Dirac}
the rescaled Dirac operator has an expansion 
\begin{eqnarray}
\boxdot & = & \left(\sum_{j=0}^{N}k^{-j/r_{y}}\boxdot_{j}\right)+k^{-2\left(N+1\right)/r_{y}}\mathrm{E}_{N+1},\;\forall N.\label{eq: Taylor expansion Dirac}
\end{eqnarray}
Here each 
\begin{align}
\boxdot_{j} & =a_{j;pq}\left(y\right)\partial_{y_{p}}\partial_{y_{q}}+b_{j;p}\left(y\right)\partial_{y_{p}}+c_{j}\left(y\right)\label{eq: operators in expansion}
\end{align}
is a ($k$-independent) self-adjoint, second-order differential operator
while each 
\begin{equation}
\mathrm{E}_{j}=\sum_{\left|\alpha\right|=N+1}y^{\alpha}\left[a_{j;pq}^{\alpha}\left(y;k\right)\partial_{y_{p}}\partial_{y_{q}}+b_{j;p}^{\alpha}\left(y;k\right)\partial_{y_{p}}+c_{j}^{\alpha}\left(y;k\right)\right]\label{eq: error operators}
\end{equation}
is a $k$-dependent self-adjoint, second-order differential operator
on $\mathbb{R}^{2}$ . Furthermore the functions appearing in \prettyref{eq: operators in expansion}
are polynomials with degrees satisfying 
\begin{align*}
\textrm{deg }a_{j}=j,\textrm{ deg }b_{j}\leq j+r_{y}-1, & \textrm{deg }c_{j}\leq j+2r_{y}-2\\
\textrm{ deg }b_{j}-\left(j-1\right)=\textrm{deg }c_{j}-j=0 & \;(\textrm{mod }2)
\end{align*}
and whose coefficients involve 
\begin{align*}
a_{j}: & \leq j-2\textrm{ derivatives of }R^{TY}\\
b_{j}: & \leq j-2\textrm{ derivatives of }R^{F},R^{\Lambda^{0,*}}\\
 & \leq j+r-2\textrm{ derivatives of }R^{L}\\
c_{j}: & \leq j-2\textrm{ derivatives of }R^{F},R^{\Lambda^{0,*}}\\
 & \leq j+r-2\textrm{ derivatives of }R^{L}
\end{align*}
while the coefficients $a_{j;pq}^{\alpha}\left(y;k\right),b_{j;p}^{\alpha}\left(y;k\right),c_{j}^{\alpha}\left(y;k\right)$
of \prettyref{eq: error operators} are uniformly (in $k$) $C^{\infty}$
bounded. Using \prettyref{eq:curvature formulas for connection},
\prettyref{eq:model Bochner in coordinates}, \prettyref{eq: model Kodaira Laplace}
and \prettyref{eq:model Lichnerowicz} the leading term of \prettyref{eq: Taylor expansion Dirac}
is computed 
\begin{equation}
\boxdot_{0}=\boxdot_{g^{TY},j_{y}^{r_{y}-2}R^{L},J^{TY}}\label{eq:leading term}
\end{equation}
in terms of the the model Kodaira Laplacian on the tangent space $TY$
\prettyref{eq: model Kodaira Laplace}.

It is now clear from \prettyref{eq:rescaled Dirac} that for $\varphi$
supported and equal to one near $0$. In light of the spectral gap
\prettyref{eq: spectral gap}, the equation \prettyref{eq: rescaling Schw kernel}
specializes to 
\begin{equation}
\tilde{\Pi}_{k}\left(y',y\right)=k^{2/r_{y}}\Pi^{\boxdot}\left(y'k^{1/r_{y}},yk^{1/r_{y}}\right)\label{eq:Bergman kernel relation}
\end{equation}
as a relation between the Bergman kernels of $\tilde{\Box}_{k}$,
$\boxdot$. Next, the expansion \prettyref{eq: Taylor expansion Dirac}
along with local elliptic estimates gives 
\[
\left(\boxdot-z\right)^{-1}-\left(\boxdot_{0}-z\right)^{-1}=O_{H_{\textrm{loc}}^{s}\rightarrow H_{\textrm{loc}}^{s+2}}\left(k^{-1/r_{y}}\left|\textrm{Im}z\right|^{-2}\right)
\]
for each $s\in\mathbb{R}$. More generally, we let $I_{j}\coloneqq\left\{ p=\left(p_{0},p_{1},\ldots\right)|p_{\alpha}\in\mathbb{N},\sum p_{\alpha}=j\right\} $denote
the set of partitions of the integer $j$ and define 
\begin{equation}
\mathtt{C}_{j}^{z}=\sum_{p\in I_{j}}\left(z-\boxdot_{0}\right)^{-1}\left[\Pi_{\alpha}\left[\boxdot_{p_{\alpha}}\left(z-\boxdot_{0}\right)^{-1}\right]\right].\label{eq: jth term kernel expansion}
\end{equation}
Then by repeated applications of the local elliptic estimate using
\prettyref{eq: Taylor expansion Dirac} we have 
\begin{equation}
\left(z-\mathrm{\boxdot}\right)^{-1}-\left(\sum_{j=0}^{N}k^{-j/r_{y}}\mathtt{C}_{j}^{z}\right)=O_{H_{\textrm{loc}}^{s}\rightarrow H_{\textrm{loc}}^{s+2}}\left(k^{-\left(N+1\right)/r_{y}}\left|\textrm{Im}z\right|^{-2Nr_{y}-2}\right),\label{eq: resolvent expansion}
\end{equation}
for each $N\in\mathbb{N},\,s\in\mathbb{R}$. A similar expansion as
\prettyref{eq: Taylor expansion Dirac} for the operator $\left(\boxdot+1\right)^{M}\left(\boxdot-z\right)$,
$M\in\mathbb{N}$, also gives 
\begin{equation}
\left(\boxdot+1\right)^{-M}\left(\boxdot-z\right)^{-1}-\sum_{j=0}^{N}k^{-j/r_{y}}\mathtt{C}_{j,M}^{z}=O_{H_{\textrm{loc}}^{s}\rightarrow H_{\textrm{loc}}^{s+2+2M}}\left(k^{-\left(N+1\right)/r_{y}}\left|\textrm{Im}z\right|^{-2Nr_{y}-2}\right)\label{eq: regularized expansion-1-1}
\end{equation}
for operators $\mathtt{C}_{j,M}^{z}=O_{H_{\textrm{loc}}^{s}\rightarrow H_{\textrm{loc}}^{s+2+2M}}\left(k^{-\left(N+1\right)/r_{y}}\left|\textrm{Im}z\right|^{-2Nr_{y}-2}\right)$,
$j=0,\ldots,N$, with 
\[
\mathtt{C}_{0,M}^{z}=\left(\hat{\Delta}_{g^{E},F,\mu}^{\left(0\right)}+1\right)^{-M}\left(\hat{\Delta}_{g^{E},F,\mu}^{\left(0\right)}-z\right)^{-1}.
\]
For $M\gg0$ sufficiently large, Sobolev's inequality gives an expansion
for the corresponding Schwartz kernels in \prettyref{eq: regularized expansion-1-1}
in $C^{l}\left(\mathbb{R}^{2}\times\mathbb{R}^{2}\right)$, $\forall l\in\mathbb{N}_{0}$.
Next, plugging the above resolvent expansion into the Helffer-Sjöstrand
formula as before gives 
\[
\left|\varphi\left(\boxdot\right)-\sum_{j=0}^{N}k^{-j/r_{y}}\mathtt{C}_{j}^{\varphi}\right|_{C^{l}\left(\mathbb{R}^{2}\times\mathbb{R}^{2}\right)}=O\left(k^{-\left(N+1\right)/r_{y}}\right)
\]
$\forall l,N\in\mathbb{N}_{0}$ and for some ($k$-independent) $\mathtt{C}_{j}^{\varphi}\in C^{\infty}\left(\mathbb{R}^{2}\times\mathbb{R}^{2}\right)$,
$j=0,1,\ldots$, with leading term $\mathtt{C}_{0}^{\varphi}=\varphi\left(\boxdot_{0}\right)=\varphi\left(\boxdot_{g^{TY},j_{y}^{r_{y}-2}R^{L},J^{TY}}\right).$
As $\varphi$ was chosen supported near $0$, the spectral gap properties
\prettyref{eq: spectral gap}, \prettyref{prop:model spectra} give
\begin{equation}
\left|\Pi^{\boxdot}-\sum_{j=0}^{N}k^{-j/r_{y}}\mathtt{C}_{j}\right|_{C^{l}\left(\mathbb{R}^{2}\times\mathbb{R}^{2}\right)}=O\left(k^{-\left(N+1\right)/r_{y}}\right)\label{eq: model Bergman expansion}
\end{equation}
for some $\mathtt{C}_{j}\in C^{\infty}\left(\mathbb{R}^{2}\times\mathbb{R}^{2}\right)$,
$j=0,1,\ldots$, with leading term $\mathtt{C}_{0}=\Pi^{\boxdot_{g^{TY},j_{y}^{r_{y}-2}R^{L},J^{TY}}}.$
The expansion is now a consequence of \prettyref{eq:bergman vs schw.},
\prettyref{eq:Bergman localization} and \prettyref{eq:Bergman kernel relation}.
Finally, in order to show that there are no odd powers of $k^{-j/r_{y}}$,
one again notes that the operators $\boxdot_{j}$ \prettyref{eq: operators in expansion}
change sign by $\left(-1\right)^{j}$ under $\delta_{-1}x\coloneqq-x$.
Thus the integral expression \prettyref{eq: jth term kernel expansion}
corresponding to $\mathtt{C}_{j}^{z}\left(0,0\right)$ changes sign
by $\left(-1\right)^{j}$ under this change of variables and must
vanish for $j$ odd. 
\end{proof}
Next we show that a pointwise expansion on the diagonal also exists
for derivatives of the Bergman kernel. In what follows we denote by
$j^{l}s/j^{l-1}s\in S^{l}T^{*}Y\otimes E$ the component of the $l$-th
jet of a section $s\in C^{\infty}\left(E\right)$ of a Hermitian vector
bundle $E$ that lies in the kernel of the natural surjection $J^{l}\left(E\right)\rightarrow J^{l-1}\left(E\right)$. 
\begin{thm}
\label{thm: expansion derivatives Bergman}For each $l\in\mathbb{N}_{0}$,
the $l$-th jet of the on-diagonal Bergman kernel has a pointwise
expansion 
\begin{equation}
j^{l}\left[\Pi_{k}\left(y,y\right)\right]/j^{l-1}\left[\Pi_{k}\left(y,y\right)\right]=k^{\left(2+l\right)/r_{y}}\left[\sum_{j=0}^{N}c_{j}\left(y\right)k^{-2j/r_{y}}\right]+O\left(k^{-\left(2N-l-1\right)/r_{y}}\right),\label{eq: derivative expansion}
\end{equation}
$\forall N\in\mathbb{N}$, in $j^{l}\textrm{End}\left(F\right)/j^{l-1}\textrm{End}\left(F\right)=S^{l}T^{*}Y\otimes\textrm{End}\left(F\right)$,
with the leading term 
\[
c_{0}\left(y\right)=j^{l}\left[\Pi^{g_{y}^{TY},j_{y}^{r_{y}-2}R^{L},J_{y}^{TY}}\left(0,0\right)\right]/j^{l-1}\left[\Pi^{g_{y}^{TY},j_{y}^{r_{y}-2}R^{L},J_{y}^{TY}}\left(0,0\right)\right]
\]
being given in terms of the $l$-th jet of the Bergman kernel of the
Kodaira Laplacian \prettyref{eq: model Kodaira Laplace} on the tangent
space at $y$. 
\end{thm}

\begin{proof}
The proof is a modification of the previous. First note that a similar
localization 
\begin{equation}
\Pi_{k}\left(\mathsf{y},\mathsf{y}\right)-\tilde{\Pi}_{k}\left(\mathsf{y},\mathsf{y}\right)=O\left(k^{-\infty}\right),\label{eq:localization in nbhd}
\end{equation}
to \prettyref{eq:Bergman localization} is valid in $C^{l}$, $\forall l\in\mathbb{N}_{0}$,
and for $\mathsf{y}$ in a uniform neighborhood of $y$. Next differentiating
\prettyref{eq:Bergman kernel relation} with $y=y'$ gives 
\begin{equation}
\partial_{\mathsf{y}}^{\alpha}\tilde{\Pi}_{k}\left(\mathsf{y},\mathsf{y}\right)=k^{\left(2+\left|\alpha\right|\right)/r_{y}}\partial_{\mathsf{y}}^{\alpha}\Pi^{\boxdot}\left(\mathsf{y}k^{1/r_{y}},\mathsf{y}k^{1/r_{y}}\right),\label{eq: rescaling derivative of Berg.}
\end{equation}
$\forall\alpha\in\mathbb{N}_{0}^{2}$. Finally, the expansion \prettyref{eq: model Bergman expansion}
being valid in $C^{l}$, $\forall l\in\mathbb{N}_{0}$, maybe differentiated
and plugged into the above with $\mathsf{y}=0$ to give the theorem. 
\end{proof}
\begin{rem}
\label{rem: Recovering positive case}The expansion \prettyref{eq:Bergmankernelexpansion}
is the same as the positive case on $Y_{2}$ (points where $r_{y}=2$)
and furthermore uniform in any $C^{l}$-topology on compact subsets
of $Y_{2}$ cf.\ \cite[Theorem 4.1.1]{Ma-Marinescu}. In particular
the first two coefficients for $y\in Y_{2}$ are given by 
\begin{align*}
c_{0}\left(y\right) & =\Pi^{g_{y}^{TY},j_{y}^{0}R^{L},J_{y}^{TY}}\left(0,0\right)=\frac{1}{2\pi}\tau^{L}\\
c_{1}\left(y\right) & =\frac{1}{16\pi}\tau^{L}\left[\kappa-\Delta\ln\tau^{L}+4\tau^{F}\right].
\end{align*}
The derivative expansion on $Y_{2}$ is also known to satisfy $c_{0}=c_{1}=\ldots=c_{\left[\frac{l-1}{2}\right]}=0$
(i.e. begins at the same leading order $k$) with the leading term
given by 
\[
c_{\left[\frac{l+1}{2}\right]}\left(y\right)=\frac{1}{2\pi}j^{l}\tau^{L}/\frac{1}{2\pi}j^{l-1}\tau^{L}.
\]
\end{rem}

\subsection{\label{subsec:Uniform-estimates-on} Uniform estimates on the Bergman
kernel }

The expansions for the Bergman kernel \prettyref{thm:Bergman kernel expansion}
and its derivatives \prettyref{thm: expansion derivatives Bergman}
are not uniform in the point on the diagonal. For applications in
the later sections we need to give uniform estimates on the Bergman
kernel. Below we set $C_{r_{1}}\coloneqq\inf_{\left|R^{V}\right|=1}\Pi^{g^{V},R^{V},J^{V}}\left(0,0\right)$
for each $0\neq R^{V}\in S^{r_{1}-2}V^{*}\otimes\Lambda^{2}V^{*}$,
$r_{1}\geq2$. Furthermore, the Bergman kernel $\Pi^{g_{y}^{TY},j_{y}^{0}R^{L},J_{y}^{TY}}\left(0,0\right)$
of the model operator \prettyref{eq: model Kodaira Laplace} is extended
(continuously) by zero from $Y_{2}$ to $Y$. 
\begin{lem}
\label{lem: uniform estimate bergman kernel} The Bergman kernel satisfies
\begin{equation}
\left[\inf_{y\in Y_{r}}\Pi^{g_{y}^{TY},j_{y}^{r-2}R^{L},J_{y}^{TY}}\left(0,0\right)\right]\left[1+o\left(1\right)\right]k^{2/r}\leq\Pi_{k}\left(y,y\right)\leq\left[\sup_{y\in Y}\Pi^{g_{y}^{TY},j_{y}^{0}R^{L},J_{y}^{TY}}\left(0,0\right)\right]\left[1+o\left(1\right)\right]k,\label{eq: uniform est. Bergman}
\end{equation}
with the $o\left(1\right)$ terms being uniform in $y\in Y$. 
\end{lem}

\begin{proof}
Note that theorem \prettyref{thm:Bergman kernel expansion} already
shows 
\begin{equation}
\Pi_{k}\left(y,y\right)\geq C_{r_{y}}\left(\left|j^{r_{y}-2}R^{L}\right|k\right)^{2/r_{y}}-c_{y}\label{eq: first non-uniform estimate}
\end{equation}
$\forall y\in Y$, with $c_{y}=c\left(\left|j^{r_{y}-2}R^{L}\left(y\right)\right|^{-1}\right)=O_{\left|j^{r_{y}-2}R^{L}\left(y\right)\right|^{-1}}\left(1\right)$
being a ($y$-dependent) constant given in terms of the norm of the
first non-vanishing jet. The norm of this jet affects the choice of
$\varrho$ needed for \prettyref{eq: curv =00003D000026 jet comp.};
which in turn affects the $C^{\infty}$-norms of the coefficients
of \prettyref{eq: error operators} via \prettyref{eq:modified connection-1}.
We first show that this estimate extends to a small ($\left|j^{r_{y}-2}R^{L}\left(y\right)\right|$-
dependent) size neighborhood of $y$. To this end, for any $\varepsilon>0$
there exists a uniform constant $c_{\varepsilon}$ depending only
on $\varepsilon$ and$\left\Vert R^{L}\right\Vert _{C^{r}}$ such
that 
\begin{equation}
\left|j^{r_{y}-2}R^{L}\left(\mathsf{y}\right)\right|\geq\left(1-\varepsilon\right)\left|j^{r_{y}-2}R^{L}\left(y\right)\right|,\label{eq: first jet comparable}
\end{equation}
$\forall\mathsf{y}\in B_{c_{\varepsilon}\left|j^{r_{y}-2}R^{L}\right|}\left(y\right).$

We begin by rewriting the model Kodaira Laplacian $\tilde{\Box}_{k}$
\prettyref{eq:model Laplace} near $y$ in terms of geodesic coordinates
centered at $\mathsf{y}$. In the region 
\[
\mathsf{y}\in B_{c_{\varepsilon}\left|j^{r_{y}-2}R^{L}\right|}\left(y\right)\cap\left\{ C_{0}\left(\left|j^{0}R^{L}\left(\mathsf{y}\right)\right|k\right)\geq k^{2/r_{y}}\Pi^{g_{y}^{TY},j_{y}^{r_{y}-2}R^{L},J_{y}^{TY}}\left(0,0\right)\right\} 
\]
a rescaling of $\tilde{\Box}_{k}$ by $\delta_{k^{-1/2}}$, now centered
at $\mathsf{y}$, shows 
\begin{align}
\Pi_{k}\left(\mathsf{y},\mathsf{y}\right) & =k\Pi^{g_{\mathsf{y}}^{TY},j_{\mathsf{y}}^{0}R^{L},J_{\mathsf{y}}^{TY}}\left(0,0\right)+O_{\left|j^{r_{y}-2}R^{L}\left(y\right)\right|^{-1}}\left(1\right)\nonumber \\
 & =k\left|j^{0}R^{L}\left(\mathsf{y}\right)\right|\Pi^{g_{\mathsf{y}}^{TY},\frac{j_{\mathsf{y}}^{0}R^{L}}{\left|j^{0}R^{L}\left(\mathsf{y}\right)\right|},J_{\mathsf{y}}^{TY}}\left(0,0\right)+O_{\left|j^{r_{y}-2}R^{L}\left(y\right)\right|^{-1}}\left(1\right)\nonumber \\
 & \geq k^{2/r_{y}}\Pi^{g_{y}^{TY},j_{y}^{r_{y}-2}R^{L},J_{y}^{TY}}\left(0,0\right)+O_{\left|j^{r_{y}-2}R^{L}\left(y\right)\right|^{-1}}\left(1\right)\label{eq:est. region 1}
\end{align}
as in \prettyref{eq: first non-uniform estimate}. Now, in the region
\begin{align*}
\mathsf{y} & \in B_{c_{\varepsilon}\left|j^{r_{y}-2}R^{L}\right|}\left(y\right)\cap\left\{ C_{1}\left(\left|j^{1}R^{L}\left(\mathsf{y}\right)/j^{0}R^{L}\left(\mathsf{y}\right)\right|k\right)^{2/3}\right.\\
 & \qquad\left.\geq k^{2/r_{y}}\Pi^{g_{y}^{TY},j_{y}^{r_{y}-2}R^{L},J_{y}^{TY}}\left(0,0\right)\geq C_{0}\left(\left|j^{0}R^{L}\left(\mathsf{y}\right)\right|k\right)\right\} 
\end{align*}
a rescaling of $\tilde{\Box}_{k}$ by $\delta_{k^{-1/3}}$ centered
at $\mathsf{y}$ similarly shows 
\begin{align}
\Pi_{k}\left(\mathsf{y},\mathsf{y}\right) & =k^{2/3}\left[1+O\left(k^{2/r-2/3}\right)\right]\Pi^{g_{\mathsf{y}}^{TY},j_{\mathsf{y}}^{1}R^{L}/j_{\mathsf{y}}^{0}R^{L},J_{\mathsf{y}}^{TY}}\left(0,0\right)\nonumber \\
 & \qquad+O_{\left|j^{r_{y}-2}R^{L}\left(y\right)\right|^{-1}}\left(1\right)\nonumber \\
 & =k^{2/3}\left[1+O\left(k^{2/r-2/3}\right)\right]\left|j_{\mathsf{y}}^{1}R^{L}/j_{\mathsf{y}}^{0}R^{L}\right|^{2/3}\Pi^{g_{\mathsf{y}}^{TY},\frac{j_{\mathsf{y}}^{1}R^{L}/j_{\mathsf{y}}^{0}R^{L}}{\left|j_{\mathsf{y}}^{1}R^{L}/j_{\mathsf{y}}^{0}R^{L}\right|},J_{\mathsf{y}}^{TY}}\left(0,0\right)\nonumber \\
 & \qquad+O_{\left|j^{r_{y}-2}R^{L}\left(y\right)\right|^{-1}}\left(1\right)\\
 & \geq\left(1-\varepsilon\right)k^{2/r_{y}}\Pi^{g_{y}^{TY},j_{y}^{r_{y}-2}R^{L},J_{y}^{TY}}\left(0,0\right)+O_{\left|j^{r_{y}-2}R^{L}\left(y\right)\right|^{-1}}\left(1\right)\label{eq: est. region 2}
\end{align}
Next, in the region 
\begin{align*}
\mathsf{y} & \in B_{c_{\varepsilon}\left|j^{r_{y}-2}R^{L}\right|}\left(y\right)\cap\left\{ C_{2}\left(\left|j^{2}R^{L}\left(\mathsf{y}\right)/j^{1}R^{L}\left(\mathsf{y}\right)\right|k\right)^{1/2}\right.\\
 & \left.\geq k^{2/r_{y}}\Pi^{g_{y}^{TY},j_{y}^{r_{y}-2}R^{L},J_{y}^{TY}}\left(0,0\right)\geq\max\left[C_{0}\left(\left|j^{0}R^{L}\left(\mathsf{y}\right)\right|k\right),C_{1}\left(\left|j^{1}R^{L}\left(\mathsf{y}\right)/j^{0}R^{L}\left(\mathsf{y}\right)\right|k\right)^{2/3}\right]\right\} 
\end{align*}
a rescaling of $\tilde{\Box}_{k}$ by $\delta_{k^{-1/4}}$ centered
at $\mathsf{y}$ shows 
\begin{align}
\Pi_{k}\left(\mathsf{y},\mathsf{y}\right) & =k^{1/2}\left[1+O\left(k^{2/r-1/2}\right)\right]\Pi^{g_{\mathsf{y}}^{TY},j_{\mathsf{y}}^{2}R^{L}/j_{\mathsf{y}}^{1}R^{L},J_{\mathsf{y}}^{TY}}\left(0,0\right)+O_{\left|j^{r_{y}-2}R^{L}\left(y\right)\right|^{-1}}\left(1\right)\nonumber \\
 & =k^{1/2}\left[1+O\left(k^{2/r-1/2}\right)\right]\left|j_{\mathsf{y}}^{2}R^{L}/j_{\mathsf{y}}^{1}R^{L}\right|^{1/2}\Pi^{g_{\mathsf{y}}^{TY},\frac{j_{\mathsf{y}}^{2}R^{L}/j_{\mathsf{y}}^{1}R^{L}}{\left|j_{\mathsf{y}}^{2}R^{L}/j_{\mathsf{y}}^{1}R^{L}\right|},J_{\mathsf{y}}^{TY}}\left(0,0\right)+O_{\left|j^{r_{y}-2}R^{L}\left(y\right)\right|^{-1}}\left(1\right)\nonumber \\
 & \geq\left(1-\varepsilon\right)k^{2/r_{y}}\Pi^{g_{y}^{TY},j_{y}^{r_{y}-2}R^{L},J_{y}^{TY}}\left(0,0\right)+O_{\left|j^{r_{y}-2}R^{L}\left(y\right)\right|^{-1}}\left(1\right)\label{eq: est. region 3}
\end{align}
Continuing in this fashion, we are finally left with the region 
\begin{align*}
\mathsf{y} & \in B_{c_{\varepsilon}\left|j^{r_{y}-2}R^{L}\right|}\left(y\right)\cap\left\{ k^{2/r_{y}}\Pi^{g_{y}^{TY},j_{y}^{r_{y}-2}R^{L},J_{y}^{TY}}\left(0,0\right)\right.\\
 & \left.\geq\max\left[C_{0}\left(\left|j^{0}R^{L}\left(\mathsf{y}\right)\right|k\right),\ldots,C_{r_{y}-3}\left(\left|j^{r_{y}-3}R^{L}\left(\mathsf{y}\right)/j^{r_{y}-4}R^{L}\left(\mathsf{y}\right)\right|k\right)^{2/\left(r_{y}-1\right)}\right]\right\} .
\end{align*}
In this region we have 
\[
\left|j^{r_{y}-2}R^{L}\left(\mathsf{y}\right)/j^{r_{y}-3}R^{L}\left(\mathsf{y}\right)\right|\geq\left(1-\varepsilon\right)\left|j^{r_{y}-2}R^{L}\left(y\right)\right|+O\left(k^{2/r_{y}-2/\left(r_{y}-1\right)}\right)
\]
following \prettyref{eq: first jet comparable} with the remainder
being uniform. A rescaling by $\delta_{k^{-1/r_{y}}}$ then giving
a similar estimate in this region, we have finally arrived at 
\[
\Pi_{k}\left(\mathsf{y},\mathsf{y}\right)\geq\left(1-\varepsilon\right)k^{2/r_{y}}\Pi^{g_{y}^{TY},j_{y}^{r_{y}-2}R^{L},J_{y}^{TY}}\left(0,0\right)+O_{\left|j^{r_{y}-2}R^{L}\left(y\right)\right|^{-1}}\left(1\right)
\]
$\forall\mathsf{y}\in B_{c_{\varepsilon}\left|j^{r_{y}-2}R^{L}\right|}\left(y\right)$.

Finally a compactness argument finds a finite set of points $\left\{ y_{j}\right\} _{j=1}^{N}$
such that the corresponding $B_{c_{\varepsilon}\left|j^{r_{y_{j}}-2}R^{L}\right|}\left(y_{j}\right)$'s
cover $Y$. This gives a uniform constant $c_{1,\varepsilon}>0$ such
that 
\[
\Pi_{k}\left(y,y\right)\geq\left(1-\varepsilon\right)\left[\inf_{y\in Y_{r}}\Pi^{g_{y}^{TY},j_{y}^{r-2}R^{L},J_{y}^{TY}}\left(0,0\right)\right]k^{2/r}-c_{1,\varepsilon}
\]
$\forall y\in Y$ , $\varepsilon>0$ proving the lower bound \prettyref{eq: uniform est. Bergman}.
The argument for the upper bound is similar. 
\end{proof}
We now prove a second lemma giving a uniform estimate on the derivatives
of the Bergman kernel. Again below, the model Bergman kernel $\Pi^{g_{y}^{TY},j_{y}^{1}R^{L}/j_{y}^{0}R^{L},J_{y}^{TY}}\left(0,0\right)$
and its relevant ratio 
\[
\frac{\left|\left[j^{l}\Pi^{g_{\mathsf{y}}^{TY},j_{\mathsf{y}}^{1}R^{L}/j_{\mathsf{y}}^{0}R^{L},J_{\mathsf{y}}^{TY}}\right]\left(0,0\right)\right|}{\Pi^{g_{\mathsf{y}}^{TY},j_{\mathsf{y}}^{1}R^{L}/j_{\mathsf{y}}^{0}R^{L},J_{\mathsf{y}}^{TY}}\left(0,0\right)}
\]
are extended (continuously) by zero from $\left\{ y|j_{y}^{1}R^{L}/j_{y}^{0}R^{L}\neq0\right\} $
to $Y$. 
\begin{lem}
\label{lem: uniform estimate on derivative Bergman}The $l$-th jet
of the Bergman kernel satisfies 
\[
\left|j^{l}\left[\Pi_{k}\left(y,y\right)\right]\right|\leq k^{l/3}\left[1+o\left(1\right)\right]\left[\sup_{y\in Y}\frac{\left|\left[j^{l}\Pi^{g_{y}^{TY},j_{y}^{1}R^{L}/j_{y}^{0}R^{L},J_{y}^{TY}}\right]\left(0,0\right)\right|}{\Pi^{g_{y}^{TY},j_{y}^{1}R^{L}/j_{y}^{0}R^{L},J_{y}^{TY}}\left(0,0\right)}\right]\Pi_{k}\left(y,y\right)
\]
with the $o\left(1\right)$ term being uniform in $y\in Y$. 
\end{lem}

\begin{proof}
The proof follows a similar argument as the previous lemma. Given
$\varepsilon>0$ we find a uniform $c_{\varepsilon}$ such that \prettyref{eq: first jet comparable}
holds for each $y\in Y$ and $\mathsf{y}\in B_{c_{\varepsilon}\left|j^{r_{y}-2}R^{L}\right|}\left(y\right)$.
Then rewrite the model Kodaira Laplacian $\tilde{\Box}_{k}$ \prettyref{eq:model Laplace}
near $y$ in terms of geodesic coordinates centered at $\mathsf{y}$.
In the region 
\[
\mathsf{y}\in B_{c_{\varepsilon}\left|j^{r_{y}-2}R^{L}\right|}\left(y\right)\cap\left\{ C_{0}\left(\left|j^{0}R^{L}\left(\mathsf{y}\right)\right|k\right)\geq k^{2/r_{y}}\Pi^{g_{y}^{TY},j_{y}^{r_{y}-2}R^{L},J_{y}^{TY}}\left(0,0\right)\right\} 
\]
a rescaling of $\tilde{\Box}_{k}$ by $\delta_{k^{-1/2}}$, now centered
at $\mathsf{y}$, shows 
\begin{align*}
\partial^{\alpha}\Pi_{k}\left(\mathsf{y},\mathsf{y}\right) & =\frac{k}{2\pi}\left(\partial^{\alpha}\tau^{L}\left(\mathsf{y}\right)\right)+O_{\left|j^{r_{y}-2}R^{L}\left(y\right)\right|^{-1}}\left(1\right)
\end{align*}
following \prettyref{rem: Recovering positive case} as $r_{\mathsf{y}}=2$.
Diving the above by \prettyref{eq:est. region 1} gives 
\begin{align*}
\frac{\left|\partial^{\alpha}\Pi_{k}\left(\mathsf{y},\mathsf{y}\right)\right|}{\Pi_{k}\left(\mathsf{y},\mathsf{y}\right)} & \leq\frac{\left|\partial^{\alpha}\tau^{L}\left(\mathsf{y}\right)\right|}{\tau^{L}\left(\mathsf{y}\right)}+O_{\left|j^{r_{y}-2}R^{L}\left(y\right)\right|^{-1}}\left(k^{-1}\right)\\
 & \leq k^{\left|\alpha\right|/3}\left[\sup_{y\in Y}\frac{\left|\left[j^{\left|\alpha\right|}\Pi^{g_{y}^{TY},j_{y}^{1}R^{L}/j_{y}^{0}R^{L},J_{y}^{TY}}\right]\left(0,0\right)\right|}{\Pi^{g_{y}^{TY},j_{y}^{1}R^{L}/j_{y}^{0}R^{L},J_{y}^{TY}}\left(0,0\right)}\right]\Pi_{k}\left(\mathsf{y},\mathsf{y}\right)\\
 & +O_{\left|j^{r_{y}-2}R^{L}\left(y\right)\right|^{-1}}\left(k^{-1}\right)
\end{align*}
Next, in the region 
\begin{align*}
\mathsf{y} & \in B_{c_{\varepsilon}\left|j^{r_{y}-2}R^{L}\right|}\left(y\right)\cap\left\{ C_{1}\left(\left|j^{1}R^{L}\left(\mathsf{y}\right)/j^{0}R^{L}\left(\mathsf{y}\right)\right|k\right)^{2/3}\right.\\
 & \left.\geq k^{2/r_{y}}\Pi^{g_{y}^{TY},j_{y}^{r_{y}-2}R^{L},J_{y}^{TY}}\left(0,0\right)\geq C_{0}\left(\left|j^{0}R^{L}\left(\mathsf{y}\right)\right|k\right)\right\} 
\end{align*}
a rescaling of $\tilde{\Box}_{k}$ by $\delta_{k^{-1/3}}$ centered
at $\mathsf{y}$ similarly shows 
\begin{align*}
\partial^{\alpha}\Pi_{k}\left(\mathsf{y},\mathsf{y}\right) & =k^{\left(2+\left|\alpha\right|\right)/3}\left[1+O\left(k^{2/r-2/3}\right)\right]\left[\partial^{\alpha}\Pi^{g_{\mathsf{y}}^{TY},j_{\mathsf{y}}^{1}R^{L}/j_{\mathsf{y}}^{0}R^{L},J_{\mathsf{y}}^{TY}}\right]\left(0,0\right)\\
 & \qquad+O_{\left|j^{r_{y}-2}R^{L}\left(y\right)\right|^{-1}}\left(k^{\left(1+\left|\alpha\right|\right)/3}\right)
\end{align*}
as in \prettyref{thm: expansion derivatives Bergman}. Dividing this
by \prettyref{eq: est. region 2} gives 
\begin{align*}
\frac{\left|\partial^{\alpha}\Pi_{k}\left(\mathsf{y},\mathsf{y}\right)\right|}{\Pi_{k}\left(\mathsf{y},\mathsf{y}\right)} & \leq k^{\left|\alpha\right|/3}\left(1+\varepsilon\right)\frac{\left|\left[\partial^{\alpha}\Pi^{g_{\mathsf{y}}^{TY},j_{\mathsf{y}}^{1}R^{L}/j_{\mathsf{y}}^{0}R^{L},J_{\mathsf{y}}^{TY}}\right]\left(0,0\right)\right|}{\left[\Pi^{g_{\mathsf{y}}^{TY},j_{\mathsf{y}}^{1}R^{L}/j_{\mathsf{y}}^{0}R^{L},J_{\mathsf{y}}^{TY}}\right]\left(0,0\right)}\\
 & \qquad+O_{\left|j^{r_{y}-2}R^{L}\left(y\right)\right|^{-1}}\left(k^{\left(\left|\alpha\right|-1\right)/3}\right)\\
 & \leq k^{\left|\alpha\right|/3}\left(1+\varepsilon\right)\left[\sup_{y\in Y}\frac{\left|\left[j^{\left|\alpha\right|}\Pi^{g_{y}^{TY},j_{y}^{1}R^{L}/j_{y}^{0}R^{L},J_{y}^{TY}}\right]\left(0,0\right)\right|}{\Pi^{g_{y}^{TY},j_{y}^{1}R^{L}/j_{y}^{0}R^{L},J_{y}^{TY}}\left(0,0\right)}\right]\\
 & \qquad+O_{\left|j^{r_{y}-2}R^{L}\left(y\right)\right|^{-1}}\left(k^{\left(\left|\alpha\right|-1\right)/3}\right).
\end{align*}
Continuing in this fashion as before eventually gives

\begin{align*}
\frac{\left|\partial^{\alpha}\Pi_{k}\left(\mathsf{y},\mathsf{y}\right)\right|}{\Pi_{k}\left(\mathsf{y},\mathsf{y}\right)} & \leq k^{\left|\alpha\right|/3}\left(1+\varepsilon\right)\left[\sup_{y\in Y}\frac{\left|\left[j^{\left|\alpha\right|}\Pi^{g_{y}^{TY},j_{y}^{1}R^{L}/j_{y}^{0}R^{L},J_{y}^{TY}}\right]\left(0,0\right)\right|}{\Pi^{g_{y}^{TY},j_{y}^{1}R^{L}/j_{y}^{0}R^{L},J_{y}^{TY}}\left(0,0\right)}\right]\\
 & \qquad+O_{\left|j^{r_{y}-2}R^{L}\left(y\right)\right|^{-1}}\left(k^{\left(\left|\alpha\right|-1\right)/3}\right)
\end{align*}
$\forall y\in Y$, $\mathsf{y}\in B_{c_{\varepsilon}\left|j^{r_{y}-2}R^{L}\right|}\left(y\right)$,
$\forall\alpha\in\mathbb{N}_{0}^{2}$. By compactness one again finds
a uniform $c_{1,\varepsilon}$ such that 
\[
\frac{\left|\partial^{\alpha}\Pi_{k}\left(y,y\right)\right|}{\Pi_{k}\left(y,y\right)}\leq k^{\left|\alpha\right|/3}\left(1+\varepsilon\right)\left[\sup_{y\in Y}\frac{\left|\left[j^{\left|\alpha\right|}\Pi^{g_{y}^{TY},j_{y}^{1}R^{L}/j_{y}^{0}R^{L},J_{y}^{TY}}\right]\left(0,0\right)\right|}{\Pi^{g_{y}^{TY},j_{y}^{1}R^{L}/j_{y}^{0}R^{L},J_{y}^{TY}}\left(0,0\right)}\right]+c_{1,\varepsilon}
\]
$\forall y\in Y$, proving the lemma. 
\end{proof}

\section{\label{subsection:Tian} Induced Fubini-Study metrics}

A theorem of Tian \cite{Tian90}, with improvements in \cite{Catlin97-Bergmankernel,Zelditch98-Bergmankernel}
(see also \cite[S 5.1.2, S 5.1.4]{Ma-Marinescu}), asserts that the
induced Fubini-Study metrics by Kodaira embeddings given by $k$th
tensor powers of a positive line bundle converge to the curvature
of the bundle as $k$ goes to infinity. In this Section we will give
a generalization for semi-positive line bundles on compact Riemann
surfaces.

Let us review first Tian's theorem. Let $(Y,J,g^{TY})$ be a compact
Hermitian manifold, $(L,h^{L})$, $(F,h^{F})$ be holomorphic Hermitian
line bundles such that $(L,h^{L})$ is positive. We endow $H^{0}(Y;F\otimes L^{k})$
with the $L^{2}$ product induced by $g^{TY}$, $h^{L}$ and $h^{F}$.
This induces a Fubini-Study metric $\omega_{FS}$ on the projective
space $\mathbb{P}\left[H^{0}\left(Y;F\otimes L^{k}\right)^{*}\right]$
and a Fubini-Study metric $h_{FS}$ on $\mathcal{O}(1)\to\mathbb{P}\left[H^{0}\left(Y;F\otimes L^{k}\right)^{*}\right]$
(see \cite[S 5.1]{Ma-Marinescu}). Since $(L,h^{L})$ is positive
the Kodaira embedding theorem shows that the Kodaira maps $\Phi_{k}:Y\rightarrow\mathbb{P}\left[H^{0}\left(Y;F\otimes L^{k}\right)^{*}\right]$
(see \prettyref{eq:Kodaira map}) are embeddings for $k\gg0$. Moreover,
the Kodaira map induces a canonical isomorphism $\Theta_{k}:F\otimes L^{k}\to\Phi_{k}^{*}\mathcal{O}(1)$
and we have (see e.g.\ \cite[(5.1.15)]{Ma-Marinescu}) 
\begin{equation}
(\Theta_{k}^{*}h_{FS})(y)=\Pi_{k}(y,y)^{-1}h^{F\otimes L^{k}}(y),\:\:y\in Y.\label{eq:hFS}
\end{equation}
This implies immediately (see e.g.\ \cite[(5.1.50)]{Ma-Marinescu})
\begin{equation}
\frac{1}{k}\Phi_{k}^{*}\omega_{FS}-\frac{i}{2\pi}R^{L}=\frac{i}{2\pi k}R^{F}-\frac{i}{2\pi k}\bar{\partial}\partial\ln\Pi_{k}\left(y,y\right).\label{eq:oFS}
\end{equation}
Applying now the Bergman kernel expansion in the positive case one
obtains Tian's theorem, which asserts that we have 
\begin{equation}
\frac{1}{k}\Phi_{k}^{*}\omega_{FS}-\frac{i}{2\pi}R^{L}=O\left(k^{-1}\right),\:\:k\to\infty,\:\:\text{in any \ensuremath{C^{\ell}}-topology}.\label{eq:Tian positive}
\end{equation}
Let us also consider the convergence of the induced Fubini-Study metric
$\Theta_{k}^{*}h_{FS}$ to the initial metric $h^{L}$. For this purpose
we fix a metric $h_{0}^{L}$ on $L$ with positive curvature. We can
then express $h^{L}=e^{-\varphi}h_{0}^{L}$, $\Theta_{k}^{*}h_{FS}=e^{-\varphi_{k}}(h_{0}^{L})^{k}\otimes h^{F}$,
where $\varphi,\varphi_{k}\in C^{\infty}(Y)$ are the global potentials
of the metrics $h$ and $\Theta_{k}^{*}h_{FS}$ with respect to $h_{0}^{L}$
and $(h_{0}^{L})^{k}\otimes h^{F}$. Note that 
\[
R^{(L,h^{L})}=R^{(L,h_{0}^{L})}+\partial\overline{\partial}\varphi,\:\:R^{(L^{k},\Theta_{k}^{*}h_{FS})}=kR^{(L,h_{0}^{L})}+R^{(F,h^{F})}+\partial\overline{\partial}\varphi_{k},
\]
and $\frac{i}{2\pi}R^{(L,\Theta_{k}^{*}h_{FS})}=\Phi_{k}^{*}\omega_{FS}$.
Then \prettyref{eq:hFS} can be written as 
\begin{equation}
\frac{1}{k}\varphi_{k}(y)-\varphi(y)=\frac{1}{k}\ln\Pi_{k}(y,y),\:\:y\in Y.\label{eq:hFS1}
\end{equation}
We obtain by \prettyref{eq:Bergmankernelexpansion} that 
\begin{equation}
\Big|\frac{1}{k}\varphi_{k}-\varphi\Big|_{C^{0}(Y)}=O\big(k^{-1}\ln k\big),\:\:k\to\infty,\label{eq:hFS2}
\end{equation}
that is, the normalized potentials of the Fubini-Study metric converge
uniformly on $Y$ to the potential of the initial metric $h^{L}$
with speed $k^{-1}\ln k$. Moreover, 
\begin{equation}
\Big|\frac{1}{k}\partial\varphi_{k}-\partial\varphi\Big|_{C^{0}(Y)}=O\big(k^{-1}\big),\:\:\Big|\frac{1}{k}\partial\overline{\partial}\varphi_{k}-\partial\overline{\partial}\varphi\Big|_{C^{0}(Y)}=O\big(k^{-1}\big),\:\:k\to\infty,\label{eq:hFS3}
\end{equation}
and we get the same bound $O\big(k^{-1}\big)$ for higher derivatives,
obtaining again \prettyref{eq:Tian positive}. Note that if $g^{TY}$
is the metric associated to $\omega=\frac{i}{2\pi}R^{L}$, then we
have a bound $O(k^{-2})$ in \prettyref{eq:Tian positive} and \prettyref{eq:hFS3}.

We return now to our situation and consider that $Y$ is a compact
Riemann surface and $(L,h^{L})$, $(F,h^{F})$ be holomorphic Hermitian
line bundles on $Y$ such that $(L,h^{L})$ is semi-positive and its
curvature vanishes at finite order. An immediate consequence of \prettyref{lem: uniform estimate bergman kernel}
is that the base locus 
\[
\textrm{Bl}\left(F\otimes L^{k}\right)\coloneqq\left\{ y\in Y|s\left(y\right)=0,\:s\in H^{0}\left(Y;F\otimes L^{k}\right)\right\} =\emptyset
\]
is empty for $k\gg0$. This shows that the subspace 
\[
\Phi_{k,y}\coloneqq\left\{ s\in H^{0}\left(Y;F\otimes L^{k}\right)|s\left(y\right)=0\right\} \subset H^{0}\left(Y;F\otimes L^{k}\right),
\]
is a hyperplane for each $y\in Y$. One may identify the Grassmanian
$\mathbb{G}\left(d_{k}-1;H^{0}\left(Y;F\otimes L^{k}\right)\right)$,
$d_{k}\coloneqq\textrm{dim}H^{0}\left(Y;F\otimes L^{k}\right)$, with
the projective space $\mathbb{P}\left[H^{0}\left(Y;F\otimes L^{k}\right)^{*}\right]$
by sending a non-zero dual element in $H^{0}\left(Y;F\otimes L^{k}\right)^{*}$
to its kernel. This now gives a well-defined Kodaira map 
\begin{align}
\Phi_{k}:Y & \rightarrow\mathbb{P}\left[H^{0}\left(Y;F\otimes L^{k}\right)^{*}\right],\nonumber \\
\Phi_{k}\left(y\right) & \coloneqq\left\{ s\in H^{0}\left(Y;F\otimes L^{k}\right)|s\left(y\right)=0\right\} .\label{eq:Kodaira map}
\end{align}
It is well known that the map is holomorphic. 
\begin{thm}
\label{thm:Tian semi-positive} Let $Y$ be a compact Riemann surface
and $(L,h^{L})$, $(F,h^{F})$ be holomorphic Hermitian line bundles
on $Y$ such that $(L,h^{L})$ is semi-positive and its curvature
vanishes at most at finite order. Then the normalized potentials of
the Fubini-Study metric converge uniformly on $Y$ to the potential
of the initial metric $h^{L}$ with speed $k^{-1}\ln k$ as in \prettyref{eq:hFS2}.
Moreover, 
\begin{equation}
\Big|\frac{1}{k}\partial\varphi_{k}-\partial\varphi\Big|_{C^{0}(Y)}\,,\:\:\Big|\frac{1}{k}\overline{\partial}\varphi_{k}-\overline{\partial}\varphi\Big|_{C^{0}(Y)}=O\big(k^{-2/3}\big),\:\:k\to\infty,\label{eq:hFS4}
\end{equation}
and 
\begin{equation}
\Big|\frac{1}{k}\partial\overline{\partial}\varphi_{k}-\partial\overline{\partial}\varphi\Big|_{C^{0}(Y)}=O\big(k^{-1/3}\big),\:\:k\to\infty,\label{eq:hFS5}
\end{equation}
especially 
\begin{equation}
\frac{1}{k}\Phi_{k}^{*}\omega_{FS}-\frac{i}{2\pi}R^{L}=O\left(k^{-1/3}\right),\:\:k\to\infty,\label{eq: uniform conv. of FS}
\end{equation}
uniformly on $Y$. On compact sets of $Y_{2}$ the estimates \prettyref{eq:Tian positive}
and \prettyref{eq:hFS3} hold. 
\end{thm}

\begin{proof}
The proof follows from \prettyref{eq:oFS}, \prettyref{eq:hFS1} and
the uniform estimate of \prettyref{lem: uniform estimate on derivative Bergman}
on the derivatives of the Bergman kernel.
\end{proof}
As we noted before, the bundle $L$ satisfying the hypotheses of \prettyref{thm:Tian semi-positive}
is ample, so for $k\gg0$ the Kodaira map is an embedding and the
induced Fubini-Study forms $\frac{1}{k}\Phi_{k}^{*}\omega_{FS}$ are
indeed metrics on $Y$. Due to the possible degeneration of the curvature
$R^{L}$ the rate of convergence in \prettyref{eq: uniform conv. of FS}
is slower than in the positive case \prettyref{eq:Tian positive}.

One can easily prove a generalization of \prettyref{thm:Tian semi-positive}
for vector bundles $(F,h^{F})$ of arbitrary rank (see \cite[S 5.1.4]{Ma-Marinescu}
for the case of a positive bundle $(L,h^{L})$). We have then Kodaira
maps $\Phi_{k}:Y\to\mathbb{G}\left(\textrm{rk}\left(F\right);H^{0}\left(Y;F\otimes L^{k}\right)^{\!*}\right)$
into the Grassmanian of $\textrm{rk}\left(F\right)$-dimensional linear
spaces of $H^{0}\left(Y;F\otimes L^{k}\right)^{\!*}$ and we introduce
the Fubini-Study metric on the Grassmannian as the curvature of the
determinant bundle of the dual of the tautological bundle (cf.\ \cite[(5.1.6)]{Ma-Marinescu}).
Then by following the proof of \cite[Theorem 5.1.17]{Ma-Marinescu}
and using \prettyref{lem: uniform estimate on derivative Bergman}
we obtain 
\begin{equation}
\frac{1}{k}\Phi_{k}^{*}\omega_{FS}-\textrm{rk}\left(F\right)\frac{i}{2\pi}R^{L}=O\left(k^{-1/3}\right),\:\:k\to\infty,\label{eq: uniform conv. of FS1}
\end{equation}
uniformly on $Y$.

\section{\label{sec:Toeplitz-operators} Toeplitz operators}

A generalization of the projector \prettyref{eq: Bergman projector}
and Bergman kernel \prettyref{eq: Bergman kernel} is given by the
notion of a Toeplitz operator. The Toeplitz operator $T_{f,k}$ operator
corresponding to a section $f\in C^{\infty}\left(Y;\textrm{End}\left(F\right)\right)$
is defined via 
\begin{align}
T_{f,k} & :C^{\infty}\left(Y;F\otimes L^{k}\right)\rightarrow C^{\infty}\left(Y;F\otimes L^{k}\right)\nonumber \\
T_{f,k} & \coloneqq\Pi_{k}f\Pi_{k},\label{eq:standard Toeplitz}
\end{align}
where $f$ denotes the operator of pointwise composition by $f$.
Each Toeplitz operator above further maps $H^{0}\left(Y;F\otimes L^{k}\right)$
to itself.

We now prove the expansion for the kernel of a Toeplitz operator generalizing
\prettyref{thm:Bergman kernel expansion}. For positive line bundles
the analogous result was proved in \cite[Theorem 2]{Charles2003}
for compact Kähler manifolds and $F=\mathbb{C}$ and in \cite[Lemma 7.2.4 and (7.4.6)]{Ma-Marinescu},
\cite[Lemma 4.6]{Ma-Marinescu2008}, in the symplectic case. 
\begin{thm}
\label{thm:Toeplitz expansion}Let $Y$ be a compact Riemann surface,
$(L,h^{L})\to Y$ a semi-positive line bundle whose curvature $R^{L}$
vanishes to finite order at any point. Let $(F,h^{F})\to Y$ be a
Hermitian holomorphic vector bundle. Then the kernel of the Toeplitz
operator \prettyref{eq:standard Toeplitz} has an on diagonal asymptotic
expansion 
\[
T_{f,k}\left(y,y\right)=k^{2/r_{y}}\left[\sum_{j=0}^{N}c_{j}\left(f,y\right)k^{-2j/r_{y}}\right]+O\left(k^{-2N/r_{y}}\right),\quad\forall N\in\mathbb{N}
\]
where the coefficients $c_{j}(f,\cdot)$ are sections of $\textrm{End}(F)$
with leading term 
\[
c_{0}\left(f,y\right)=\Pi^{g_{y}^{TY},R_{y}^{TY},J_{y}^{TY}}\left(0,0\right)f\left(y\right).
\]
\end{thm}

\begin{proof}
Firstly from the definition \prettyref{eq:standard Toeplitz} and
the localization/rescaling properties \prettyref{eq:Bergman localization},
\prettyref{eq:Bergman kernel relation} one has 
\begin{align}
T_{f,k}\left(y,y\right) & =\int_{Y}dy'\,\Pi_{k}\left(y,y'\right)f\left(y'\right)\Pi_{k}\left(y',y\right)\nonumber \\
 & =\int_{B_{\varepsilon}\left(y\right)}dy'\,\tilde{\Pi}_{k}\left(0,y'\right)f\left(y'\right)\tilde{\Pi}_{k}\left(y',0\right)+O\left(k^{-\infty}\right)\nonumber \\
 & =\int_{B_{\varepsilon}\left(y\right)}dy'\,k^{4/r_{y}}\Pi^{\boxdot}\left(0,y'k^{1/r_{y}}\right)f\left(y'\right)\Pi^{\boxdot}\left(y'k^{1/r_{y}},0\right)+O\left(k^{-\infty}\right)\nonumber \\
 & =\int_{k^{1/r_{y}}B_{\varepsilon}\left(y\right)}dy'\,k^{2/r_{y}}\Pi^{\boxdot}\left(0,y'\right)f\left(y'k^{-1/r_{y}}\right)\Pi^{\boxdot}\left(y',0\right)+O\left(k^{-\infty}\right).\label{eq: first step Toeplitz expansino}
\end{align}
Next as in \prettyref{sec:Model-operators}, $\varphi\left(\boxdot\right)\left(.,0\right)\in\mathcal{S}\left(V\right)$
for $\varphi\in\mathcal{S}\left(\mathbb{R}\right)$ in the Schwartz
class. Thus plugging \prettyref{eq: model Bergman expansion} and
a Taylor expansion 
\[
f\left(y'k^{-1/r_{y}}\right)=\sum_{\left|\alpha\right|\leq N+1}\frac{1}{\alpha!}\left(y'\right)^{\alpha}k^{-\alpha/r_{y}}f^{\left(\alpha\right)}\left(0\right)+O\left(k^{-\left(N+1\right)/r_{y}}\right)
\]
into \prettyref{eq: first step Toeplitz expansino} above gives the
result with the leading term again coming from \prettyref{eq:leading term}.
Finally and as in the proof of \prettyref{thm:Bergman kernel expansion},
there are no odd powers of $k^{-j/r_{y}}$ as the corresponding coefficients
are given by odd integrals (the integrands change sign by $\left(-1\right)^{j}$
under $\delta_{-1}x\coloneqq-x$) which are zero. 
\end{proof}
We now show that the Toeplitz operators \prettyref{eq:standard Toeplitz}
can be composed up to highest order generalizing the results of \cite{Bordemann-Meinrenken-Schlichenmaier94}
in the Kähler case and $F=\mathbb{C}$ and \cite[Theorems 7.4.1--2]{Ma-Marinescu},
\cite[Theorems 1.1 and 4.19]{Ma-Marinescu2008} in the symplectic
case. 
\begin{thm}
Given $f,g\in C^{\infty}\left(Y;\textrm{End}\left(F\right)\right)$,
the Toeplitz operators \prettyref{eq:standard Toeplitz} satisfy 
\begin{align}
\lim_{k\rightarrow\infty}\left\Vert T_{f,k}\right\Vert  & =\left\Vert f\right\Vert _{\infty}\coloneqq\sup_{\substack{y\in Y\\
u\in F_{y}\setminus0
}
}\frac{|f(y)u|_{h^{F}}}{|u|_{h^{F}}}\,,\label{eq: norm Toeplitz}\\
T_{f,k}T_{g,k} & =T_{fg,k}+O_{L^{2}\rightarrow L^{2}}\left(k^{-1/r}\right).\label{eq: leading comp. Toeplit}
\end{align}
\end{thm}

\begin{proof}
The first part of \prettyref{eq: norm Toeplitz} is similar to the
positive case. Firstly, $\left\Vert T_{f,k}\right\Vert \leq\left\Vert f\right\Vert _{\infty}$
is clear from the definition \prettyref{eq:standard Toeplitz}. For
the lower bound, let us consider $y\in Y_{2}$ where the curvature
is non-vanishing and $u\in F_{y}$, $|u|_{h^{F}}=1$. It follows from
the proof of \cite[Theorem 7.4.2]{Ma-Marinescu} (see also \cite[Proposition 5.2, (5.40), Remark 5.7]{Barron-Ma-Marinescu-Pinsonnault2014})
that 
\begin{equation}
|f\left(y\right)(u)|_{h^{F}}+O_{y,u}\left(k^{-1/2}\right)\leq\left\Vert T_{f,k}\right\Vert .\label{eq: lower bound Toeplitz norm 1}
\end{equation}
If $\left\Vert f\right\Vert _{\infty}=|f\left(y_{0}\right)(u_{0})|_{h^{F}}$
is attained at a point $y_{0}\in Y_{2}$, it follows immediately from
\prettyref{eq: lower bound Toeplitz norm 1} that 
\[
\left\Vert f\right\Vert _{\infty}+O\left(k^{-1/2}\right)\leq\left\Vert T_{f,k}\right\Vert ,
\]
so one obtains the lower bound. Next let $\left\Vert f\right\Vert _{\infty}=|f\left(y_{0}\right)(u_{0})|_{h^{F}}$
be attained at $y_{0}\in Y\setminus Y_{2}$, a vanishing point of
the curvature. As $Y\setminus Y_{2}\subset Y$ is open and dense one
may find for any $\varepsilon>0$ a point $y_{\varepsilon}\in Y\setminus Y_{2}$
and $u_{\varepsilon}\in F_{y_{\varepsilon}}$, $|u_{\varepsilon}|_{h^{F}}=1$,
with $\left\Vert f\right\Vert _{\infty}-\varepsilon\leq|f\left(y_{\varepsilon}\right)(u_{\varepsilon})|_{h^{F}}$.
Combined with \prettyref{eq: lower bound Toeplitz norm 1} this gives
\begin{align*}
\left\Vert f\right\Vert _{\infty}-\varepsilon+O_{\varepsilon}\left(k^{-1/2}\right) & \leq\left\Vert T_{f,k}\right\Vert ,\quad\textrm{and }\\
\left\Vert f\right\Vert _{\infty}-\varepsilon & \leq\liminf_{k\rightarrow\infty}\left\Vert T_{f,k}\right\Vert .
\end{align*}
Since $\varepsilon>0$ is arbitrary, this implies $\left\Vert f\right\Vert _{\infty}\leq\liminf_{k\rightarrow\infty}\left\Vert T_{f,k}\right\Vert $
proving the lower bound.


Next, to prove the composition expansion \prettyref{eq: leading comp. Toeplit}
it suffices to prove a uniform kernel estimate 
\[
\left\Vert \left[T_{f,k}T_{g,k}-T_{fg,k}\right]\left(.,y\right)\right\Vert _{L^{2}}=O\left(k^{-1/r}\right),\quad\forall y\in Y.
\]
To this end we again compute in geodesic chart centered at $y$ 
\begin{align*}
T_{f,k}T_{g,k}\left(.,0\right) & =\int_{Y\times Y}dy_{1}dy_{2}\,\Pi_{k}\left(.,y_{1}\right)f\left(y_{1}\right)\Pi_{k}\left(y_{1},y_{2}\right)g\left(y_{2}\right)\Pi_{k}\left(y_{2},0\right)\\
 & =O_{L^{2}}\left(k^{-\infty}\right)+\int_{B_{\varepsilon}\left(y_{2}\right)}dy_{1}\int_{B_{\varepsilon}\left(y\right)}dy_{2}\,\tilde{\Pi}_{k}\left(.,y_{1}\right)f\left(y_{1}\right)\tilde{\Pi}_{k}\left(y_{1},y_{2}\right)g\left(y_{2}\right)\tilde{\Pi}_{k}\left(y_{2},0\right)\\
 & =O_{L^{2}}\left(k^{-\infty}\right)+\int_{B_{\varepsilon}\left(y_{2}\right)}dy_{1}\int_{B_{\varepsilon}\left(y\right)}dy_{2}k^{6/r_{y}}\left\{ \Pi^{\boxdot}\left(k^{1/r_{y}}.,k^{1/r_{y}}y_{1}\right)\right.\\
 & \qquad\left.f\left(y_{1}\right)\Pi^{\boxdot}\left(k^{1/r_{y}}y_{1},k^{1/r_{y}}y_{2}\right)g\left(y_{2}\right)\Pi^{\boxdot}\left(k^{1/r_{y}}y_{2},0\right)\right\} \\
 & =O_{L^{2}}\left(k^{-\infty}\right)+\int_{k^{1/r_{y}}B_{\varepsilon}\left(y_{2}\right)}dy_{1}\int_{k^{1/r_{y}}B_{\varepsilon}\left(y\right)}dy_{2}k^{2/r_{y}}\left\{ \Pi^{\boxdot}\left(.,y_{1}\right)\right.\\
 & \qquad\left.f\left(y_{1}k^{-1/r_{y}}\right)\Pi^{\boxdot}\left(y_{1},y_{2}\right)g\left(y_{2}k^{-1/r_{y}}\right)\Pi^{\boxdot}\left(y_{2},0\right)\right\} \\
 & =O_{L^{2}}\left(k^{-1/r_{y}}\right)+\int_{k^{1/r_{y}}B_{\varepsilon}\left(y_{2}\right)}dy_{1}\int_{k^{1/r_{y}}B_{\varepsilon}\left(y\right)}dy_{2}k^{2/r_{y}}\left\{ \Pi^{\boxdot}\left(.,y_{1}\right)\right.\\
 & \qquad\left.\Pi^{\boxdot}\left(y_{1},y_{2}\right)fg\left(y_{2}k^{-1/r_{y}}\right)\Pi^{\boxdot}\left(y_{2},0\right)\right\} \\
 & =O_{L^{2}}\left(k^{-1/r_{y}}\right)+\int_{B_{\varepsilon}\left(y_{2}\right)}dy_{1}\int_{B_{\varepsilon}\left(y\right)}dy_{2}\,\tilde{\Pi}_{k}\left(.,y_{1}\right)\tilde{\Pi}_{k}\left(y_{1},y_{2}\right)fg\left(y_{2}\right)\tilde{\Pi}_{k}\left(y_{2},0\right)\\
 & =O_{L^{2}}\left(k^{-1/r_{y}}\right)+T_{fg,k}
\end{align*}
with all remainders being uniform in $y\in Y$. Above we have again
used the localization/rescaling properties \prettyref{eq:Bergman localization},
\prettyref{eq:Bergman kernel relation} as well as the first order
Taylor expansion $f\left(y_{1}k^{-1/r_{y}}\right)=f\left(y_{2}k^{-1/r_{y}}\right)+O_{\left\Vert f\right\Vert _{C^{1}}}\left(k^{-1/r_{y}}\right)$. 
\end{proof}
\begin{rem}
Similar to the previous remark \prettyref{rem: Recovering positive case},
we can recover the usual algebra properties of Toeplitz operators
when $f,g$ are compactly supported on the set $Y_{2}$ where the
curvature $R^{L}$ is positive. In particular we define a generalized
Toeplitz operator to be a sequence of operators $T_{k}:L^{2}(Y,F\otimes L^{k})\longrightarrow L^{2}(Y,F\otimes L^{k})$,
$k\in\mathbb{N}$, such that there exist $K\Subset Y_{2}$, $h_{j}\in C_{c}^{\infty}\left(K;\textrm{End}\left(F\right)\right)$,
$C_{j}>0$, $j=0,1,2,\ldots$ satisfying 
\begin{equation}
\Big\| T_{k}-\sum_{j=0}^{N}k^{-j}T_{h_{j},k}\Big\|\leqslant C_{N}\,k^{-N-1},\quad\forall N\in\mathbb{N}.\label{gen Toeplitz operator}
\end{equation}
Then this class is closed under composition and one may define a formal
star product on $C_{c}^{\infty}\left(Y_{2}\right)\left[\left[k^{-1}\right]\right]$
, via 
\begin{align*}
f\ast_{k^{-1}}g & =\sum_{j=0}^{\infty}C_{j}\left(f,g\right)k^{-j}\in C_{c}^{\infty}\left(Y_{2}\right)\left[\left[k^{-1}\right]\right]\quad\textrm{where}\\
T_{f,k}\circ T_{g,k} & \sim\sum_{j=0}^{\infty}T_{C_{j}\left(f,g\right)}k^{-j},
\end{align*}
(cf. \cite{Bordemann-Meinrenken-Schlichenmaier94,Charles2003,Ma-Marinescu2008}).
Furthermore 
\begin{align*}
T_{f,k}\circ T_{g,k} & =T_{fg,k}+O_{L^{2}\rightarrow L^{2}}\left(k^{-1}\right)\\{}
[T_{f,k}\,,T_{g,k}] & =\frac{i}{k}T_{\{f,g\},k}+O_{L^{2}\rightarrow L^{2}}(k^{-2})
\end{align*}
$\forall f,g\in C_{c}^{\infty}\left(Y_{2};\textrm{End}\left(F\right)\right)$,
with $\{\cdot,\cdot\}$ being the Poisson bracket on the Kähler manifold
$(Y_{2},iR^{L})$. 
\end{rem}

Finally we address the asymptotics of the spectral measure of the
Toeplitz operator \prettyref{eq:standard Toeplitz}, called Szeg\H{o}-type
limit formulas \cite{Boutet-Guillemin81,Guillemin79}. The spectral
measure of $T_{f,k}$ is defined via 
\begin{equation}
u_{f,k}\left(s\right)\coloneqq\sum_{\lambda\in\textrm{Spec}\left(T_{f,k}\right)}\delta\left(s-\lambda\right)\in\mathcal{S}'\left(\mathbb{R}_{s}\right).\label{eq:spectral density Toeplitz}
\end{equation}
We now have the following asymptotic formula. 
\begin{thm}
The spectral measure \prettyref{eq:spectral density Toeplitz} satisfies
\begin{equation}
u_{f,k}\sim\frac{k}{2\pi}f_{*}R^{L}\label{eq: spectral distribution Toeplitz}
\end{equation}
in the distributional sense as $k\rightarrow\infty$. 
\end{thm}

\begin{proof}
Since $\textrm{Spec}\left(T_{f,k}\right)\subset\big[-\left\Vert f\right\Vert _{\infty},\left\Vert f\right\Vert _{\infty}\big]$
by \prettyref{eq: norm Toeplitz}, the equation \prettyref{eq: spectral distribution Toeplitz}
is equivalent to 
\[
\textrm{tr }\varphi\left(T_{f,k}\right)=\sum_{\lambda\in\textrm{Spec}
\left(T_{f,k}\right)}\varphi\left(\lambda\right)
\sim\frac{k}{2\pi}\int_{Y}\left[\varphi\circ f\right]R^{L},
\]
for all $\varphi\in C_{c}^{\infty}\big(\!-\left\Vert f\right
\Vert _{\infty}-1,\left\Vert f\right\Vert _{\infty}+1\big)$.
We first prove that the trace of a Toeplitz operator \prettyref{eq:standard Toeplitz}
satisfies the asymptotics 
\begin{equation}
\textrm{tr }T_{f,k}\sim\frac{k}{2\pi}\int_{Y}fR^{L}.
\label{eq: trace asymp. of Toeplitz operator}
\end{equation}
To this end first note that the expansion of \prettyref{thm:Toeplitz expansion}
is uniform on compact subsets $K\subset Y_{2}$ while 
$\left|T_{f,k}\left(y,y\right)\right|=O\left(k\right)$
uniformly in $y\in Y$ as in \prettyref{lem: uniform estimate bergman kernel}.
Further, as with \cite[Proposition 7]{Marinescu-Savale18},
$Y_{\geq3}$ is a closed subset of a hypersurface and has measure zero. Then with
$K_{j}\subset Y_{2}$, $j=1,2,\ldots$\,, being a sequence of compact
subsets satisfying $K_{j}\subset K_{j+1}$, 
$\cap_{j=1}^{\infty}K_{j}=Y_{\geq3}$,
one may then breakup the trace integral 
\begin{align*}
\frac{1}{k}\textrm{tr}T_{f,k} & =
\frac{1}{k}\int_{K_{j}}\textrm{tr }T_{f,k}\left(y,y\right)+
\frac{1}{k}\int_{Y\setminus K_{j}}\textrm{tr }T_{f,k}
\left(y,y\right)\\
 & =\frac{1}{2\pi}\int_{K_{j}}fR^{L}+
 O_{j}\left(\frac{1}{k}\right)+O\left(\mu\left(Y\setminus K_{j}\right)\right)
\end{align*}
from which \prettyref{eq: trace asymp. of Toeplitz operator} 
follows on knowing $\frac{1}{2\pi}\int_{K_{j}}fR^{L}
\rightarrow\frac{1}{2\pi}\int_{Y}fR^{L}$,
$\mu\left(Y\setminus K_{j}\right)\rightarrow0$ as $j\rightarrow\infty.$

Following this one has 
\[
\textrm{tr }T_{f,k}^{l}=\textrm{tr }T_{f^{l},k}+O_{f}\left(k^{1-1/r}\right)
\]
for all $l\in\mathbb{N}$ from \prettyref{eq: hol. Euler characteristic},
\prettyref{eq: leading comp. Toeplit}. A polynomial approximation
of the compactly supported function $\varphi\in C_{c}^{\infty}\big(\!-\left\Vert f\right\Vert _{\infty}-1,\left\Vert f\right\Vert _{\infty}+1\big)$
then gives 
\begin{align*}
\textrm{tr }\varphi\left(T_{f,k}\right) & =\textrm{tr }T_{\varphi\circ f,k}+o\left(k\right)\\
 & =\frac{k}{2\pi}\int_{Y}\left[\varphi\circ f\right]R^{L}+o\left(k\right)
\end{align*}
by \prettyref{eq: trace asymp. of Toeplitz operator} as required. 
\end{proof}
The analogous result for projective manifolds endowed with the restriction
of the hyperplane bundle was originally proved in \cite[Theorem 13.13]{Boutet-Guillemin81},
\cite{Guillemin79} and for arbitrary positive line bundles in \cite{Berndtsson2003},
see also \cite{Lindholm2001}. In \cite[Theorem 1.6]{Hsiao-Marinescu2017}
the asymptotics \prettyref{eq: trace asymp. of Toeplitz operator}
are proved for a semi-classical spectral function of the Kodaira Laplacian
on an arbitrary manifold.

\subsection{Branched coverings}

We now consider Toeplitz operators and their composition in a particular
case of semipositive line bundles. Namely, those that arise from pullbacks
along branched coverings. Here $f:Y\rightarrow Y_{0}$ is a branched
covering of a Riemann surface $Y_{0}$ with branch points $\left\{ y_{1},\ldots,y_{M}\right\} \subset Y$.
The Hermitian holomorphic line bundle on $Y$ is pulled back $\left(L,h^{L}\right)=\left(f^{*}L_{0},f^{*}h^{L_{0}}\right)$
from one on $Y_{0}$. If $\left(L_{0},h^{L_{0}}\right)$ is assumed
positive, then $\left(L,h^{L}\right)$ is semi-positive with curvature
vanishing at the branch points. In particular, near a branch point
$y\in Y$ of local degree $\frac{r}{2}$ one may find holomorphic
geodesic coordinate such that the curvature is given by 
$R^{L}=\frac{r^{2}}{4}\left(z\bar{z}\right)^{r/2-1}
R_{f\left(y\right)}^{L_{0}}+O\left(y^{r-1}\right)$.
We denote for simplicity $R_0:=R_{f\left(y\right)}^{L_{0}}$.

The leading term of \prettyref{eq:Bergmankernelexpansion} is given
by the model Bergman kernel $\Pi^{\boxdot_{0}}\left(0,0\right)$ of
the operator 
\begin{align}
\boxdot_{0} & =bb^{\dagger},\quad\textrm{ for }\label{eq:model operator branched covering}\\
b^{\dagger} & =2\partial_{\bar{z}}+a,\nonumber \\
a & =\frac{r}{4}z\left(z\bar{z}\right)^{r/2-1}R_{0}.\label{eq:model operator branched covering II}
\end{align}
We first compute this model Bergman kernel.
\begin{lem}
The model Bergman kernel corresponding to the model operator \prettyref{eq:model operator branched covering}
at a branch point is given by 
\begin{align}
\Pi^{\boxdot_{0}}\left(z,z'\right) & =\frac{re^{-2\left[\Phi\left(z\right)+\Phi\left(z'\right)\right]}R_{0}^{\frac{2}{r}}}{2\pi}G\left(R_{0}^{\frac{2}{r}}z\overline{z'}\right)\quad\textrm{where }\label{eq:model Bergman kernel}\\
\Phi\left(z\right) & \coloneqq\frac{1}{4}\left(z\bar{z}\right)^{r/2}R_{0}\quad\textrm{ and}\label{eq:Potential}\\
G\left(x\right) & \coloneqq\sum_{\alpha=0}^{\frac{r}{2}-1}\frac{x^{\alpha}}{\Gamma\left(\frac{2\left(\alpha+1\right)}{r}\right)}+x^{\frac{r}{2}-1}e^{x^{\frac{r}{2}}}\left[\sum_{\alpha=0}^{\frac{r}{2}-2}\frac{\Gamma\left(\frac{2\left(\alpha+1\right)}{r}\right)-\Gamma\left(\frac{2\left(\alpha+1\right)}{r},x^{\frac{r}{2}}\right)}{\Gamma\left(\frac{2\left(\alpha+1\right)}{r}\right)/\left(\frac{2\left(\alpha+1\right)}{r}-1\right)}\right]\label{eq:function G}
\end{align}
is given in terms of the incomplete gamma function. 
\end{lem}

\begin{proof}
From the formulas \prettyref{eq:model operator branched covering II},
an orthonormal basis for $\textrm{ker}\left(\boxdot_{0}\right)$ is
easily found to be
\begin{align*}
s_{\alpha} & \coloneqq\left(\frac{1}{2\pi}\frac{r}{\Gamma\left(\frac{2\left(\alpha+1\right)}{r}\right)}R_{0}^{\frac{2\left(\alpha+1\right)}{r}}\right)^{1/2}z^{\alpha}e^{-\Phi},\quad\alpha\in\mathbb{N}_{0},\\
\textrm{with}\quad\Phi & \coloneqq\frac{1}{4}\left(z\bar{z}\right)^{r/2}R_{0}.
\end{align*}
From here the model Bergman kernel is computed 
\begin{align}
\Pi^{\boxdot_{0}}\left(z,z'\right) & =\sum_{\alpha\in\mathbb{N}_{0}}s_{\alpha}\left(z\right)\overline{s_{\alpha}\left(z'\right)}\nonumber \\
 & =\frac{1}{2\pi}\sum_{\alpha\in\mathbb{N}_{0}}\frac{r}{\Gamma\left(\frac{2\left(\alpha+1\right)}{r}\right)}R_{0}^{\frac{2\left(\alpha+1\right)}{r}}\left(z\overline{z'}\right)^{\alpha}e^{-2\Phi}.\label{eq:Bergman kernel orthogonal basis}
\end{align}
To compute the above in a closed form, consider the series
\begin{align*}
F\left(y\right) & \coloneqq\sum_{\alpha=0}^{\infty}\frac{y^{\frac{\alpha+1}{s}-1}}{\Gamma\left(\frac{\alpha+1}{s}\right)}\\
 & =\sum_{\alpha=0}^{s-1}\frac{y^{\frac{\alpha+1}{s}-1}}{\Gamma\left(\frac{\alpha+1}{s}\right)}+\underbrace{\sum_{\alpha=s}^{\infty}\frac{y^{\frac{\alpha+1}{s}-1}}{\Gamma\left(\frac{\alpha+1}{s}\right)}}_{F_{0}\left(y\right)\coloneqq},
\end{align*}
for $s=\frac{r}{2}$. Differentiating the second term in the series
gives $F_{0}'\left(y\right)=F_{0}\left(y\right)+\sum_{\alpha=0}^{s-2}\left(\frac{\alpha+1}{s}-1\right)\frac{y^{\frac{\alpha+1}{s}-1}}{\Gamma\left(\frac{\alpha+1}{s}\right)}$.
Which is an ODE that can be solved with the initial condition $F_{0}\left(0\right)=0$
to give 
\begin{align}
F_{0}\left(y\right) & =\sum_{\alpha=s}^{\infty}\frac{y^{\frac{\alpha+1}{s}-1}}{\Gamma\left(\frac{\alpha+1}{s}\right)}=e^{y}\left[\sum_{\alpha=0}^{s-2}\frac{\Gamma\left(\frac{\alpha+1}{s}\right)-\Gamma\left(\frac{\alpha+1}{s},y\right)}{\Gamma\left(\frac{\alpha+1}{s}\right)/\left(\frac{\alpha+1}{s}-1\right)}\right]\nonumber \\
\textrm{in terms of }\Gamma\left(a,z\right) & \coloneqq\int_{z}^{\infty}t^{a-1}e^{-t}dt,\quad\textrm{Re}\left(z\right)>0,\label{eq:important series}
\end{align}
the incomplete gamma function. Thus in particular we have computed
$F\left(y\right)\coloneqq y^{\frac{1}{s}-1}G\left(y^{\frac{1}{s}}\right)$
\prettyref{eq:function G}. Finally noting from \prettyref{eq:Bergman kernel orthogonal basis}
that 
\[
\Pi^{\boxdot_{0}}\left(z,z'\right)=\frac{re^{-2\Phi}R_{0}^{\frac{2}{r}}}{2\pi}x^{s-1}F\left(x^{s}\right),
\]
for $x=R_{0}^{\frac{2}{r}}z\overline{z'}$, completes the proof.
\end{proof}
This gives the first term of the expansion
\[
c_{0}\left(y\right)=\Pi^{\boxdot_{0}}\left(0,0\right)=\frac{1}{2\pi}\frac{r}{\Gamma\left(\frac{2}{r}\right)}R_{0}^{\frac{2}{r}}
\]
at the vanishing/branch point $y$ in this example.

\section{\label{subsec:Random-sections} Random sections}

In this section we generalize the results of \cite{Shiffman-Zelditch99}
to the semi-positive case considered here. Let us consider Hermitian
holomorphic line bundles $(L,h^{L})$ and $(F,h^{F})$ on a compact
Riemann surface $Y$. To state the result first note that the natural
metric on $H^{0}\left(Y;F\otimes L^{k}\right)$ arising from $g^{TY}$,
$h^{F}$and $h^{L}$ gives rise to a probability density $\mu_{k}$
on the sphere 
\[
SH^{0}\left(Y;F\otimes L^{k}\right)\coloneqq\left\{ s\in H^{0}\left(Y;F\otimes L^{k}\right)|\left\Vert s\right\Vert =1\right\} ,
\]
of finite dimension $\chi\left(Y;F\otimes L^{k}\right)-1$ \prettyref{eq: hol. Euler characteristic}.
We now define the product probability space $\left(\Omega,\mu\right)\coloneqq\left(\Pi_{k=1}^{\infty}SH^{0}\left(Y;F\otimes L^{k}\right),\Pi_{k=1}^{\infty}\mu_{k}\right)$.
To a random sequence of sections $s=\left(s_{k}\right)_{k\in\mathbb{N}}\in\Omega$
given by this probability density, we then associate the random sequence
of zero divisors $Z_{s_{k}}=\left\{ s_{k}=0\right\} $ and view it
as a random sequence of currents of integration in $\Omega_{0,0}\left(Y\right)$.
We now have the following. 
\begin{thm}
\label{thm: random section} Let $(L,h^{L})$ and $(F,h^{F})$ be
Hermitian holomorphic line bundles on a compact Riemann surface $Y$
and assume that $(L,h^{L})$ is semi-positive line bundle and its
curvature $R^{L}$ vanishes to finite order at any point. Then for
$\mu$-almost all $s=\left(s_{k}\right)_{k\in\mathbb{N}}\in\Omega$,
the sequence of currents 
\[
\frac{1}{k}Z_{s_{k}}\rightharpoonup\frac{i}{2\pi}R^{L}
\]
converges weakly to the semi-positive curvature form. 
\end{thm}

\begin{proof}
The proof follows \cite{Ma-Marinescu} Thm 5.3.3 with some modifications
which we point out below. With $\Phi_{k}$ denoting the Kodaira map
\prettyref{eq:Kodaira map}, we first have 
\begin{equation}
\mathbb{E}\left[Z_{s_{k}}\right]=\Phi_{k}^{*}\left(\omega_{FS}\right)\label{eq: prob. Poincare Lelong}
\end{equation}
as in \cite{Ma-Marinescu} Thm 5.3.1. For a given $\varphi\in\Omega_{0,0}\left(Y\right)$,
one has 
\[
\left\langle \frac{1}{k}Z_{s_{k}}-\frac{i}{2\pi}R^{L},\varphi\right\rangle =\left\langle \frac{1}{k}Z_{s_{k}}-\frac{1}{k}\Phi_{k}^{*}\left(\omega_{FS}\right),\varphi\right\rangle +O\left(k^{-1/3}\left\Vert \varphi\right\Vert _{C^{0}}\right)
\]
following \prettyref{eq: uniform conv. of FS} and it thus suffices
to show $Y^{\varphi}\left(s_{k}\right)\rightarrow0$, $\mu$-almost
surely with 
\begin{align*}
Y^{\varphi}\left(s_{k}\right) & \coloneqq\left\langle \frac{1}{k}Z_{s_{k}}-\frac{1}{k}\Phi_{k}^{*}\left(\omega_{FS}\right),\varphi\right\rangle 
\end{align*}
being the given random variable. But \prettyref{eq: prob. Poincare Lelong}
gives 
\begin{align*}
\mathbb{E}\left[\left|Y^{\varphi}\left(s_{k}\right)\right|^{2}\right] & =\frac{1}{k^{2}}\mathbb{E}\left[\left\langle Z_{s_{k}},\varphi\right\rangle ^{2}\right]-\frac{1}{k^{2}}\mathbb{E}\left[\left\langle \Phi_{k}^{*}\left(\omega_{FS}\right),\varphi\right\rangle ^{2}\right]\\
 & =O\left(k^{-2}\right)
\end{align*}
as in \cite{Ma-Marinescu} Thm 5.3.3. Thus $\int_{\Omega}d\mu\left[\sum_{k=1}^{\infty}\left|Y^{\varphi}\left(s_{k}\right)\right|^{2}\right]<\infty$
proving the theorem. 
\end{proof}
The above result may be alternatively obtained using $L^{2}$ estimates
for the $\bar{\partial}$-equation of a modified positive metric as
in \cite[S 4]{Dinh-Ma-Marinescu2016}.
\begin{example}
(Random polynomials) The last theorem has an interesting specialization
to random polynomials. To this end, let $Y=\mathbb{CP}^{1}=\mathbb{C}_{w}^{2}\setminus\left\{ 0\right\} /\mathbb{C}^{*}$
with homogeneous coordinates $\left[w_{0}:w_{1}\right]$. A semi-positive
curvature form for each even $r\geq2$, is given by 
\begin{equation}
\begin{split}\omega_{r}\coloneqq & \frac{i}{2\pi}\partial\bar{\partial}\ln\left(\left|w_{0}\right|^{r}+\left|w_{1}\right|^{r}\right)\\
= & \frac{i}{2\pi}\frac{r^{2}}{4}\frac{\left|w_{0}\right|^{r-2}\left|w_{1}\right|^{r-2}}{\left(\left|w_{0}\right|^{r}+\left|w_{1}\right|^{r}\right)^{2}}dw_{0}\wedge dw_{1},
\end{split}
\label{eq:semi-positiveform}
\end{equation}
which has two vanishing points at the north/south poles of order $r-2$.
This is the curvature form on the hyperplane line bundle $L=\mathcal{O}\left(1\right)$
for the metric with potential $\varphi=\ln\left(\left|w_{0}\right|^{r}+\left|w_{1}\right|^{r}\right)$.
An orthogonal basis for $H^{0}\left(X,L^{k}\right)$ is given by $s_{\alpha}\coloneqq z^{\alpha}$,
$0\leq\alpha\leq k$, in terms of the affine coordinate $z=w_{0}/w_{1}$
on the chart $\left\{ w_{1}\neq0\right\} $ and a $\mathbb{C}^{*}$
invariant trivialization of $L$. The normalization is now given by
\begin{align*}
\left\Vert s_{\alpha}\right\Vert ^{2} & =\frac{1}{2\pi}\frac{r^{2}}{4}\int_{\mathbb{C}}\frac{\left|z\right|^{2\alpha+r-2}}{\left(1+\left|z\right|^{r}\right)^{k+2}}\\
 & =\frac{1}{\frac{2}{r}\left(k+1\right)\begin{pmatrix}k\\
\frac{2}{r}\alpha
\end{pmatrix}}
\end{align*}
with the binomial coefficient $\begin{pmatrix}k\\
\frac{2}{r}\alpha
\end{pmatrix}=\frac{\Gamma\left(k+1\right)}{\Gamma\left(\frac{2}{r}\alpha+1\right)\Gamma\left(k-\frac{2}{r}\alpha\right)}$ given in terms of the Gamma function. We have now arrived at the
following. 
\end{example}

\begin{cor}
For each even $r\geq2$, let 
\[
p_{k}\left(z\right)=\sum_{\alpha=0}^{k}c_{\alpha}\sqrt{\begin{pmatrix}k\\
\frac{2}{r}\alpha
\end{pmatrix}}z^{\alpha}
\]
be a random polynomial of degree $k$ with the coefficients $c_{\alpha}$
being standard i.i.d.\ Gaussian variables. The distribution of its
roots converges in probability 
\[
\frac{1}{k}Z_{p_{k}}\rightharpoonup\frac{1}{2\pi}\frac{r^{2}}{4}\frac{\left|z\right|^{r-2}}{\left(1+\left|z\right|^{r}\right)^{2}}\,\cdot
\]
\end{cor}

The above theorem interpolates between the case of $SU\left(2\right)$/elliptic
polynomials ($r=2$) \cite{Bogomolny-Bohigas-Lebouef96} and the case
of Kac polynomials ($r=\infty$) \cite{Hammersley56,Kac43,Shepp-Vanderbei95}.
For recent results on the distribution of zeroes of more general classes
of random polynomials we refer to \cite{Bayraktar-Coman-Herrmann-Marinescu2018,Bloom-Levenberg2015,Kabluchko-Zaporozhets2014}.

\section{\label{sec:Holomorphic-torsion} Holomorphic torsion}

In this section we give an asymptotic result for the holomorphic torsion
of the semi-positive line bundle $L$ generalizing that of \cite{Bismut-Vasserot}
(see also \cite[S 5.5]{Ma-Marinescu}). First recall that the holomorphic
torsion of $L$ is defined in terms of the zeta function 
\begin{equation}
\zeta_{k}\left(s\right)\coloneqq\frac{1}{\Gamma\left(s\right)}\int_{0}^{\infty}dt\,t^{s-1}\textrm{tr}\left[e^{-t\Box_{k}^{1}}\right],\quad\textrm{Re}\left(s\right)>1.\label{eq:zeta fn.}
\end{equation}
The above converges absolutely and defines a holomorphic function
of $s\in\mathbb{C}$ in this region. It possesses a meromorphic extension
to $\mathbb{C}$ with no pole at zero and the holomorphic torsion
is defined to be $\mathcal{T}_{k}\coloneqq\exp\left\{ -\frac{1}{2}\zeta'_{k}\left(0\right)\right\} $.

Next, with $\tau^{L}$, $\omega\left(R^{L}\right)$ as in \prettyref{eq:formulas Clifford-1}
and $t>0$, set 
\begin{equation}
R_{t}\left(y\right)\coloneqq\begin{cases}
\frac{1}{2\pi}\tau^{L}\left(1-e^{-t\tau^{L}}\right)^{-1}e^{-t\omega\left(R^{L}\right)}; & \tau^{L}\left(y\right)>0\\
\frac{1}{2\pi}\frac{1}{t}; & \tau^{L}\left(y\right)=0.
\end{cases}\label{eq:heat coefficients Rt}
\end{equation}
Note that the above defines a smooth endomorphism $R_{t}\left(y\right)\in C^{\infty}\left(Y;\textrm{End}\left(\Lambda^{0,*}\right)\right)$.
Further, let $A_{j}\in C^{\infty}\left(Y;\textrm{End}\left(\Lambda^{0,*}\right)\right)$
be such that 
\begin{equation}
\rho_{t}^{N}\coloneqq R_{t}\left(y\right)-\sum_{j=-1}^{N}A_{j}\left(y\right)t^{j}=O\left(t^{N+1}\right).\label{eq:small time trace exp. D0}
\end{equation}
We now prove the following uniform small time asymptotic expansion
for the heat kernel. 
\begin{prop}
There exist $A_{k,j}\in C^{\infty}\left(Y;\textrm{End}\left(\Lambda^{0,*}\right)\right)$,
$j=-1,0,1,\ldots$, satisfying $A_{k,j}-A_{j}=O\left(k^{-1}\right)$,
such that for each $t>0$

\begin{equation}
\left|k^{-1}e^{-\frac{t}{2k}D_{k}^{2}}\left(y,y\right)-\sum_{j=-1}^{N}A_{k,j}\left(y\right)t^{j}-\rho_{t}^{N}\right|=O\left(t^{N+1}k^{-1}\right)\label{eq:small time exp. heat}
\end{equation}
uniformly in $y\in Y$, $k\in\mathbb{N}$. 
\end{prop}

\begin{proof}
We again work in the geodesic coordinates and local orthonormal frames
centered at $y\in Y$ introduced in Section \prettyref{subsec:Bergman-kernel exp.}.
With $\tilde{D}_{k}$ as in \prettyref{eq: local Dirac}, similar
localization estimates as in Section \prettyref{subsec:Bergman-kernel exp.}
(cf.\ also Lemma 1.6.5 in \cite{Ma-Marinescu}) give 
\[
e^{-\frac{t}{k}D_{k}^{2}}\left(y,y\right)-e^{-\frac{t}{k}\tilde{D}_{k}^{2}}\left(0,0\right)=O\left(e^{-\frac{\varepsilon^{2}}{32t}k}\right)
\]
uniformly in $t>0$, $y\in Y$ and $k$. It then suffices to consider
the small time asymptotics of $e^{-\frac{t}{k}\tilde{D}_{k}^{2}}\left(0,0\right)$.
We again introduce the rescaling $\delta_{k^{-1/2}}y=k^{-1/2}y$,
under which
\[
\left(\delta_{k^{-1/2}}\right)_{*}\tilde{D}_{k}^{2}=k\Big(\underbrace{\boxdot_{0}+k^{-1}E_{1}}_{\eqqcolon\boxdot}\Big),
\]
with $\boxdot_{0}=bb^{\dagger}+\tau_{y}\bar{w}i_{\bar{w}}$; $b\coloneqq-2\partial_{z}+\frac{1}{2}\tau_{y}\bar{z}$
and where $E_{0}=a_{pq}\left(y;k\right)\partial_{y_{p}}\partial_{y_{q}}+b_{p}\left(y;k\right)\partial_{y_{p}}+c\left(y;k\right)$
is a second order operator whose coefficient functions are uniformly
(in $k$) $C^{\infty}$ bounded (cf.\ \cite{Ma-Marinescu} Thms.
4.1.7 and 4.1.25). Again $e^{-\frac{t}{k}\tilde{D}_{k}^{2}}\left(0,0\right)=ke^{-t\boxdot}\left(0,0\right)$
and following a standard small time heat kernel expansion of an elliptic
operator \cite{Greiner71,Seeley67} one has 
\begin{equation}
e^{-t\boxdot}\left(y_{1},y_{2}\right)=\frac{1}{4\pi t}e^{-\frac{1}{4t}d^{2}\left(y_{1},y_{2}\right)}\left[\sum_{j=-1}^{N}A_{k,j}t^{j}\right]+\rho_{k,t}^{N}\label{eq:small time exp k dep.}
\end{equation}
with $d\left(y_{1},y_{2}\right)$ denoting the distance function for
the metric $\tilde{g}^{TY}$ on $\mathbb{R}^{2}$. Moreover 
\begin{equation}
A_{k,j}\left(y_{1},y_{2}\right)=A_{j}\left(y_{1},y_{2}\right)+O\left(k^{-1}\right)\label{eq:asymptotic heat coeff}
\end{equation}
where $A_{j}$ denotes an analogous term in the small time expansion
of $e^{-t\boxdot_{0}}$ satisfying $A_{j}\left(0,0\right)=A_{j}$\prettyref{eq:small time trace exp. D0}
(\cite[(1.6.68)]{Ma-Marinescu}). Finally, the remainders in \prettyref{eq:small time trace exp. D0},
\prettyref{eq:small time exp k dep.} being given by 
\[
\rho_{t}^{N}=-\int_{0}^{t}ds\,e^{-\left(t-s\right)\boxdot_{0}}s^{N}\left(\boxdot_{0}A_{N}\right),\quad\rho_{k,t}^{N}=-\int_{0}^{t}ds\,e^{-\left(t-s\right)\boxdot}s^{N}\left(\boxdot A_{k,N}\right),
\]
the proposition now follows from \prettyref{eq:asymptotic heat coeff}
along with $e^{-t\boxdot}=e^{-t\boxdot_{0}}+O\left(k^{-1}\right)$
uniformly in $t>0$. 
\end{proof}
We now prove the the asymptotic result for holomorphic torsion. Below
we denote by $x\ln x$ the continuous extension of this function from
$\mathbb{R}_{>0}$ to $\mathbb{R}_{\geq0}$ (i.e. taking the value
zero at the origin). 
\begin{thm}
The holomorphic torsion satisfies the asymptotics 
\[
\ln\mathcal{T}_{k}\coloneqq-\frac{1}{2}\zeta'_{k}\left(0\right)=-k\ln k\left[\frac{\tau^{L}}{8\pi}\right]-k\left[\frac{\tau^{L}}{8\pi}\ln\left(\frac{\tau^{L}}{2\pi}\right)\right]+o\left(k\right)
\]
as $k\rightarrow\infty$. 
\end{thm}

\begin{proof}
First define the rescaled zeta function $\tilde{\zeta}_{k}\left(s\right)\coloneqq\frac{k^{-1}}{\Gamma\left(s\right)}\int_{0}^{\infty}dt\,t^{s-1}\textrm{tr}\left[e^{-\frac{t}{k}\Box_{k}^{1}}\right]=k^{-1}k^{s}\zeta_{k}\left(s\right)$
satisfying 
\begin{equation}
\zeta_{k}'\left(0\right)=k\tilde{\zeta}_{k}'\left(0\right)-\left(k\ln k\right)\tilde{\zeta}_{k}\left(0\right).\label{eq: zeta vs resc. zeta}
\end{equation}
With $a_{k,j}\coloneqq\int_{Y}\textrm{tr}\left[A_{k,j}\right]dy$,
$j=-1,0,\ldots$, and the analytic continuation of the zeta function
being given in terms of the heat trace, one has 
\begin{align}
\tilde{\zeta}_{k}\left(0\right) & =a_{k,0}\rightarrow\int_{Y}dy\,\textrm{tr}\left[A_{0}\right],\label{eq: zeta conv.}\\
\tilde{\zeta}_{k}'\left(0\right) & =\underbrace{\int_{0}^{T}dt\,t^{-1}\left\{ k^{-1}\textrm{tr}\left[e^{-\frac{t}{k}\Box_{k}^{1}}\right]-a_{k,-1}t^{-1}-a_{k,0}\right\} }_{=\int_{0}^{T}dt\,t^{-1}\rho_{t}^{0}+O\left(\frac{T}{k}\right)}\nonumber \\
 & \quad+\int_{T}^{\infty}dt\,t^{-1}k^{-1}\textrm{tr}\left[e^{-\frac{t}{k}\Box_{k}^{1}}\right]\nonumber \\
 & \quad-a_{k,-1}T^{-1}+\Gamma'\left(1\right)a_{k,0}\label{eq: zeta in heat tr.}
\end{align}
following \prettyref{eq:small time exp. heat}.

Choosing $T=k^{1-2/r}$, gives 
\begin{align*}
t^{-1}k^{-1}\textrm{tr}\left[e^{-\frac{t}{k}\Box_{k}^{1}}\right] & \leq e^{-\frac{\left(t-1\right)}{k}\left[c_{1}k^{2/r}-c_{2}\right]}t^{-1}k^{-1}\textrm{tr}\left[e^{-\frac{1}{k}\Box_{k}^{1}}\right]\\
 & \leq Ct^{-1}k^{-1}e^{-\frac{\left(t-1\right)}{k}\left[c_{1}k^{2/r}-c_{2}\right]},\quad t\geq T,
\end{align*}
on account of \prettyref{eq:spectral gap Kodaira}, \prettyref{eq:small time exp. heat}.
The last expression having a uniformly in $k$ bounded integral on
$\left[T,\infty\right)$, dominated convergence gives 
\begin{align}
\tilde{\zeta}_{k}'\left(0\right)\longrightarrow & \int_{Y}dy\,\alpha\left(y\right),\quad\textrm{ where }\nonumber \\
\alpha\left(y\right)\coloneqq & \int_{0}^{T}dt\,t^{-1}\left\{ \textrm{tr}\left[R_{t}\left(y\right)\right]-\textrm{tr}\left[A_{-1}\right]t^{-1}-\textrm{tr}\left[A_{0}\right]\right\} \nonumber \\
 & +\int_{T}^{\infty}dt\,t^{-1}\textrm{tr}\left[R_{t}\left(y\right)\right]\nonumber \\
 & -\textrm{tr}\left[A_{-1}\right]t^{-1}+\Gamma'\left(1\right)\textrm{tr}\left[A_{0}\right].\label{eq: resc. zeta conv.}
\end{align}
Finally, using \prettyref{eq:heat coefficients Rt} one has 
\begin{align}
\textrm{tr}\left[A_{0}\right] & =-\frac{\tau^{L}}{4\pi}\nonumber \\
\alpha\left(y\right) & =\frac{\tau^{L}}{4\pi}\ln\left(\frac{\tau^{L}}{2\pi}\right)\label{eq: zeta limits}
\end{align}
with again the extension of the function $x\ln x$ to the origin being
given by continuity to be zero as before. The proposition now follows
from putting together \prettyref{eq: zeta vs resc. zeta}, \prettyref{eq: zeta conv.},
\prettyref{eq: zeta in heat tr.}, \prettyref{eq: resc. zeta conv.}
and \prettyref{eq: zeta limits}. 
\end{proof}

\appendix

\section{\label{sec:Model-operators}Model operators}

Here we define certain model Bochner/Kodaira Laplacians and Dirac
operators acting on a vector space $V$. First the Bochner Laplacian
is intrinsically associated to a triple $\left(V,g^{V},R^{V}\right)$
with metric $g^{V}$ and tensor $0\neq R^{V}\in S^{r-2}V^{*}\otimes\Lambda^{2}V^{*}$,
$r\geq2$. We say that tensor $R^{V}$ is nondegenerate if 
\begin{equation}
S^{r-s-2}V^{*}\otimes\Lambda^{2}V^{*}\ni i_{v}^{s}\left(R^{V}\right)=0,\,\forall s\leq r-2\,\implies T_{y}Y\ni v=0.\label{eq:non-degeneracy condition}
\end{equation}
Above $i^{s}$ denotes the $s$-fold contraction of the symmetric
part of $R^{V}$.

For $v_{1}\in V$, $v_{2}\in T_{v_{1}}V=V$, contraction of the antisymmetric
part (denoted by $\iota$) of $R^{V}$ gives $\iota_{v_{2}}R^{V}\in S^{r-2}V^{*}\otimes V^{*}$.
The contraction may then be evaluated $\left(\iota_{v_{2}}R^{V}\right)\left(v_{1}\right)$
at $v_{1}\in V$, i.e. viewed as a homogeneous degree $r-1$ polynomial
function on $V$. The tensor $R^{V}$ now determines a one form $a^{R^{V}}\in\Omega^{1}\left(V\right)$
via 
\begin{equation}
a_{v_{1}}^{R^{V}}\left(v_{2}\right)\coloneqq\int_{0}^{1}d\rho\left(\iota_{v_{2}}R^{V}\right)\left(\rho v_{1}\right)=\frac{1}{r}\left(\iota_{v_{2}}R^{V}\right)\left(v_{1}\right),\label{eq:connection from tensot}
\end{equation}
which we may view as a unitary connection $\nabla^{R^{V}}=d+ia^{R^{V}}$
on a trivial Hermitian vector bundle $E$ of arbitrary rank over $V$.
The curvature of this connection is clearly $R^{V}$ now viewed as
a homogeneous degree $r-2$ polynomial function on $V$ valued in
$\Lambda^{2}V^{*}$. This now gives the model Bochner Laplacian 
\begin{equation}
\Delta_{g^{V},R^{V}}\coloneqq\left(\nabla^{R^{V}}\right)^{*}\nabla^{R^{V}}:C^{\infty}\left(V;E\right)\rightarrow C^{\infty}\left(V;E\right).\label{eq:model Bochner Laplacian}
\end{equation}
An orthonormal basis $\left\{ e_{1},e_{2},\ldots,e_{n}\right\} $,
determines components $R_{pq,\alpha}\coloneqq R^{V}\left(e^{\odot\alpha};e_{p},e_{q}\right)\neq0$,
$\alpha\in\mathbb{N}_{0}^{n-1}$, $\left|\alpha\right|=r-2$, as well
as linear coordinates $\left(y_{1},\ldots,y_{n}\right)$ on $V$.
The connection form in these coordinates is given by $a_{p}^{R^{V}}=\frac{i}{r}y^{q}y^{\alpha}R_{pq,\alpha}.$
While the model Laplacian \prettyref{eq:model Bochner Laplacian}
is given 
\begin{equation}
\Delta_{g^{V},R^{V}}=-\sum_{q=1}^{n}\left(\partial_{y_{p}}+\frac{i}{r}y^{q}y^{\alpha}R_{pq,\alpha}\right)^{2}.\label{eq:model Bochner in coordinates}
\end{equation}
As in \prettyref{eq: Fourier decomposition Laplacian}, the above
may now be related to the (nilpotent) sR Laplacian on the the product
$S_{\theta}^{1}\times V$ given by 
\begin{equation}
\hat{\Delta}_{g^{V},R^{V}}\coloneqq-\sum_{q=1}^{n}\left(\partial_{y_{p}}+\frac{i}{r}y^{q}y^{\alpha}R_{pq,\alpha}\partial_{\theta}\right)^{2},\label{eq:nilpotent Laplacian}
\end{equation}
and corresponding to the sR structure $\left(S_{\theta}^{1}\times V,\ker\left(d\theta+a^{R^{V}}\right),\text{\ensuremath{\pi}}^{*}g^{V},\,d\theta\textrm{vol}g^{V}\right)$
where the sR metric corresponds to $g^{V}$ under the natural projection
$\pi:S_{\theta}^{1}\times V\rightarrow V$. Note that the above differs
from the usual nilpotent approximation of the sR Laplacian since it
acts on the product with $S^{1}$. As \prettyref{eq:Fourier mode heat kernel},
the heat kernels of \prettyref{eq:model Bochner Laplacian}, \prettyref{eq:nilpotent Laplacian}
are now related 
\begin{equation}
e^{-t\Delta_{g^{V},R^{V}}}\left(y,y'\right)=\int e^{-i\theta}e^{-t\hat{\Delta}_{g^{V},R^{V}}}\left(y,0;y',\theta\right)d\theta.\label{eq: model heat ker. reln.}
\end{equation}

Next, assume that the vector space $V$ of even dimension and additionally
equipped with an orthogonal endomorphism $J^{V}\in O\left(V\right)$;
$\left(J^{V}\right)^{2}=-1$. This gives rise to a (linear) integrable
almost complex structure on $V$ , a decomposition $V\otimes\mathbb{C}=V^{1,0}\oplus V^{0,1}$
into $\pm i$ eigenspaces of $J$ and a Clifford multiplication endomorphism
$c:V\rightarrow\textrm{End}\left(\Lambda^{*}V^{0,1}\right)$. We further
assume that $R^{V}$ is a $\left(1,1\right)$ form with respect to
$J$ (i.e. $S^{k}V^{*}\ni R^{V}\left(w_{1},w_{2}\right)=0$, $\forall w_{1},w_{2}\in V^{1,0}$).
The $\left(0,1\right)$ part of the connection form \prettyref{eq:connection from tensot}
then gives a holomorphic structure on the trivial Hermitian line bundle
$\mathbb{C}$ with holomorphic derivative $\bar{\partial}_{\mathbb{C}}=\bar{\partial}+\left(a^{V}\right)^{0,1}$.
One may now define the Kodaira Dirac and Laplace operators, intrinsically
associated to the tuple $\left(V,g^{V},R^{V},J^{V}\right)$, via 
\begin{align}
D_{g^{V},R^{V},J^{V}}\coloneqq & \sqrt{2}\left(\bar{\partial}_{\mathbb{C}}+\bar{\partial}_{\mathbb{C}}^{*}\right)\label{eq:model Kodaira Dirac}\\
\Box_{g^{V},R^{V},J^{V}}\coloneqq & \frac{1}{2}\left(D_{g^{V},R^{V},J^{V}}\right)^{2}\label{eq: model Kodaira Laplace}
\end{align}
acting on $C^{\infty}\left(V;\Lambda^{*}V^{0,1}\right)$. The above
\prettyref{eq:model Bochner Laplacian}, \prettyref{eq: model Kodaira Laplace}
are related by the Lichnerowicz formula 
\begin{equation}
\Box_{g^{V},R^{V},J^{V}}=\Delta_{g^{V},R^{V}}+c\left(R^{V}\right)\label{eq:model Lichnerowicz}
\end{equation}
where $c\left(R^{V}\right)=\sum_{p<q}R_{pq}^{i_{1}\ldots i_{r-2}}y_{i_{1}\ldots}y_{i_{r-2}}c\left(e_{p}\right)c\left(e_{q}\right)$.
We may choose complex orthonormal basis $\left\{ w_{j}\right\} _{j=1}^{m}$
of $V^{1,0}$ that diagonalizes the tensor $R^{V}$: $R^{V}\left(w_{i},\bar{w}_{j}\right)=\delta_{ij}R_{j\bar{j}}$;
$R_{i\bar{j}}\in S^{r-2}V^{*}$. This gives complex coordinates on
$V$ in which \prettyref{eq: model Kodaira Laplace} may be written
as 
\begin{align}
\Box_{g^{V},R^{V},J^{V}} & =\sum_{q=1}^{\textrm{dim}V/2}b_{j}b_{j}^{\dagger}+2\left(\partial_{z_{j}}a_{j}+\partial_{\bar{z}_{j}}\bar{a}_{j}\right)\bar{w}_{j}i_{\bar{w}_{j}}\label{eq:model Kodaira Lap.}\\
b_{j} & =-2\partial_{z_{j}}+\bar{a}_{j}\nonumber \\
b_{j}^{\dagger} & =2\partial_{\bar{z}_{j}}+a_{j}\nonumber \\
a_{j} & =\frac{1}{r}R_{j\bar{j}}z_{j}\nonumber 
\end{align}
with each $R_{j\bar{j}}\left(z\right)$, $1\leq j\leq\textrm{dim}V/2$,
being a real homogeneous function of order $r-2$.

Being symmetric with respect to the standard Euclidean density and
semi-bounded below, both $\Delta_{g^{V},R^{V}}$ and $\Box^{V}$ are
essentially self-adjoint on $L^{2}$. The domains of their unique
self-adjoint extensions are 
\[
\begin{split}\textrm{Dom}\left(\Delta_{g^{V},R^{V}}\right) & =\left\{ \psi\in L^{2}|\Delta_{g^{V},R^{V}}\psi\in L^{2}\right\} ,\\
\textrm{Dom}\left(\Box_{g^{V},R^{V},J^{V}}\right) & =\left\{ \psi\in L^{2}|\Box_{g^{V},R^{V},J^{V}}\psi\in L^{2}\right\} ,
\end{split}
\]
respectively. 
We shall need the following information regarding their spectrum. 
\begin{prop}
\label{prop:model spectra}For some $c>0$, one has $\textrm{Spec}\left(\Delta_{g^{V},R^{V}}\right)\subset\left[c,\infty\right)$.
For $R^{V}$ satisfying the non-degeneracy condition \prettyref{eq:non-degeneracy condition}
one has $\textrm{EssSpec}\left(\Delta_{g^{V},R^{V}}\right)=\emptyset$.
Finally, for $\textrm{dim}V=2$ with $R^{V}\left(w,\bar{w}\right)\geq0$,
$\forall w\in V^{1,0}$ semi-positive one has $\textrm{Spec}\left(\Box_{g^{V},R^{V},J^{V}}\right)\subset\left\{ 0\right\} \cup\left[c,\infty\right).$ 
\end{prop}

\begin{proof}
The proof is similar to those of Proposition \prettyref{prop: gen. spectral gap}
and Corollary \prettyref{cor: spectral gap Dirac}. Introduce the
deformed Laplacian $\Delta_{k}\coloneqq\Delta_{g^{V},kR^{V}}$ obtained
by rescaling the tensor $R^{V}$. From \prettyref{eq:model Bochner in coordinates}
$\Delta_{k}=k^{2/r}\mathscr{R}\Delta_{g^{V},R^{V}}\mathscr{R}^{-1}$
are conjugate under the rescaling $\mathscr{R}:C^{\infty}\left(V;E\right)\rightarrow C^{\infty}\left(V;E\right)$,
$\left(\mathscr{R}u\right)\left(x\right)\coloneqq u\left(yk^{1/r}\right)$
implying 
\begin{align}
\textrm{Spec}\left(\Delta_{k}\right) & =k^{2/r}\textrm{Spec}\left(\Delta_{g^{V},R^{V}}\right)\nonumber \\
\textrm{EssSpec}\left(\Delta_{k}\right) & =k^{2/r}\textrm{EssSpec}\left(\Delta_{g^{V},R^{V}}\right)\label{eq: rescaling of spectra}
\end{align}
By an argument similar to Proposition \prettyref{prop: gen. spectral gap},
one has $\textrm{Spec}\left(\Delta_{k}\right)\subset\left[c_{1}k^{2/r}-c_{2},\infty\right)$
for some $c_{1},c_{2}>0$ for $R^{V}\neq0$. From here $\textrm{Spec}\left(\Delta_{g^{V},R^{V}}\right)\subset\left[c,\infty\right)$
follows. Next, under the non-degeneracy condition, the order of vanishing
of the curvature homogeneous curvature $R^{V}$ (of the homogeneous
connection $a^{R^{V}}$\prettyref{eq:connection from tensot}) is
seen to be maximal at the origin: $\textrm{ord}_{y}\left(R^{V}\right)<r-2$
for $y\neq0$. Following a similar sub-elliptic estimate \prettyref{eq:local subelliptic estimate}
on $V\times S_{\theta}^{1}$ as in Proposition \prettyref{prop: gen. spectral gap},
we have 
\[
k^{2/\left(r-1\right)}\left\Vert u\right\Vert ^{2}\leq C\left[\left\langle \Delta_{k}u,u\right\rangle +\left\Vert u\right\Vert _{L^{2}}^{2}\right],\quad\forall u\in C_{c}^{\infty}\left(V\setminus B_{1}\left(0\right)\right),
\]
holds on the complement of the unit ball centered at the origin. Combining
the above with Persson's characterization of the essential spectrum
(cf.\ \cite{Persson1960}, \cite[Chapter 3]{Agmon-book-1982}) 
\[
\textrm{EssSpec}\left(\Delta_{k}\right)=\sup_{R}\inf_{\substack{\left\Vert u\right\Vert =1\\
u\in C_{c}^{\infty}\left(V\setminus B_{R}\left(0\right)\right)
}
}\left\langle \Delta_{k}u,u\right\rangle ,
\]
we have $\textrm{EssSpec}\left(\Delta_{k}\right)\subset\left[c_{1}k^{2/\left(r-1\right)}-c_{2},\infty\right)$.
From here and using \prettyref{eq: rescaling of spectra}, $\textrm{EssSpec}\left(\Delta_{g^{V},R^{V}}\right)=\emptyset$
follows.

The proof of the final part is similar following $k^{2/r}\textrm{Spec}\left(\Box_{g^{V},R^{V},J^{V}}\right)=\textrm{Spec}\left(\Box_{g^{V},kR^{V},J^{V}}\right)=\textrm{Spec}\left(\Box_{k}\right)\subset\left\{ 0\right\} \cup\left[c_{1}k^{2/r}-c_{2},\infty\right),$
$\Box_{k}\coloneqq\Box_{g^{V},kR^{V},J^{V}}$, by an argument similar
to Corollary \prettyref{cor: spectral gap Dirac}. 
\end{proof}
Next, the heat $e^{-t\Delta_{g^{V},R^{V}}}$, $e^{-t\Box_{g^{V},R^{V},J^{V}}}$
and wave $e^{it\sqrt{\Delta_{g^{V},R^{V}}}}$, $e^{it\sqrt{\Box_{g^{V},R^{V},J^{V}}}}$
operators being well-defined by functional calculus, a finite propagation
type argument as in \prettyref{lem: Localization lemma} gives $\varphi\left(\Delta_{g^{V},R^{V}}\right)\left(.,0\right)\in\mathcal{S}\left(V\right)$,
$\varphi\left(\Box_{g^{V},R^{V},J^{V}}\right)\left(.,0\right)\in\mathcal{S}\left(V\right)$
are in the Schwartz class for $\varphi\in\mathcal{S}\left(\mathbb{R}\right)$.
Further, when $\textrm{EssSpec}\left(\Delta_{g^{V},R^{V}}\right)=\emptyset$
any eigenfunction of $\Delta_{g^{V},R^{V}}$ also lies in $\mathcal{S}\left(V\right)$.
Finally, on choosing $\varphi$ supported close to the origin, the
Schwartz kernel $\Pi^{_{g^{V},R^{V},J^{V}}}\left(.,0\right)\in\mathcal{S}\left(V\right)$
of the projector $\Pi^{_{g^{V},R^{V},J^{V}}}$ onto the kernel of
$\Box_{g^{V},R^{V},J^{V}}$ is also of Schwartz class.

We now state another proposition regarding the heat kernel of $\Delta_{g^{V},R^{V}}$.
Below we denote $\lambda_{0}\left(\Delta_{g^{V},R^{V}}\right)\coloneqq\textrm{inf Spec}\left(\Delta_{g^{V},R^{V}}\right)$
. 
\begin{prop}
\label{prop:large time heat =00003D000026 bottom of spec}For each
$\varepsilon>0$ there exist $t,R>0$ such that the heat kernel 
\[
\frac{\int_{B_{R}\left(0\right)}dx\left[\Delta_{g^{V},R^{V}}e^{-t\Delta_{g^{V},R^{V}}}\right]\left(x,x\right)}{\int_{B_{R}\left(0\right)}dx\,e^{-t\Delta_{g^{V},R^{V}}}\left(x,x\right)}\leq\lambda_{0}\left(\Delta_{g^{V},R^{V}}\right)+\varepsilon
\]
\end{prop}

\begin{proof}
Setting $P\coloneqq\Delta_{g^{V},R^{V}}-\lambda_{0}\left(\Delta_{g^{V},R^{V}}\right)$
it suffices to show 
\[
\frac{\int_{B_{R}\left(0\right)}dx\left[Pe^{-tP}\right]\left(x,x\right)}{\int_{B_{R}\left(0\right)}dx\,e^{-tP}\left(x,x\right)}\leq\varepsilon
\]
for some $t,R>0$. With $\Pi_{\left[0,x\right]}^{P}$ denoting the
spectral projector onto $\left[0,x\right]$, we split the numerator
\[
\int_{B_{R}\left(0\right)}dx\left[Pe^{-tP}\right]\left(x,x\right)=\int_{B_{R}\left(0\right)}dx\left[\Pi_{\left[0,4\varepsilon\right]}^{P}Pe^{-tP}\right]\left(x,x\right)+\int_{B_{R}\left(0\right)}dx\left[\left(1-\Pi_{\left[0,4\varepsilon\right]}^{P}\right)Pe^{-tP}\right]\left(x,x\right).
\]
From $P\geq0$, $\Pi_{\left[0,4\varepsilon\right]}^{P}Pe^{-tP}\leq4\varepsilon e^{-tP}$
and $\left(1-\Pi_{\left[0,4\varepsilon\right]}^{P}\right)Pe^{-tP}\leq ce^{-3\varepsilon t}$,
$\forall t\geq1$, we may bound 
\begin{equation}
\frac{\int_{B_{R}\left(0\right)}dx\left[Pe^{-tP}\right]\left(x,x\right)}{\int_{B_{R}\left(0\right)}dx\,e^{-tP}\left(x,x\right)}\leq4\varepsilon+\frac{ce^{-3\varepsilon t}R^{n-1}}{\int_{B_{R}\left(0\right)}dx\,e^{-tP}\left(x,x\right)}\label{eq:breakup est.}
\end{equation}
$\forall R,t\geq1$. Next, as $0\in\textrm{Spec}\left(P\right)$ there
exists $\left\Vert \psi_{\varepsilon}\right\Vert _{L^{2}}=1$, $\left\Vert P\psi_{\varepsilon}\right\Vert _{L^{2}}\leq\varepsilon$.
It now follows that $\left\Vert \psi_{\varepsilon}-\Pi_{\left[0,2\varepsilon\right]}^{P}\psi_{\varepsilon}\right\Vert \leq\frac{1}{2}$
and hence 
\[
\begin{split}\frac{1}{2} & =-\frac{1}{4}+\int_{B_{R_{\varepsilon}}\left(0\right)}dx\left|\psi_{\varepsilon}\left(x\right)\right|^{2}\leq\int_{B_{R_{\varepsilon}}\left(0\right)}dx\left|\int dy\Pi_{\left[0,2\varepsilon\right]}^{P}\left(x,y\right)\psi_{\varepsilon}\left(y\right)\right|^{2}\\
 & \leq\int_{B_{R_{\varepsilon}}\left(0\right)}dx\left(\int dy\Pi_{\left[0,2\varepsilon\right]}^{P}\left(x,y\right)\Pi_{\left[0,2\varepsilon\right]}^{P}\left(y,x\right)\right)=\int_{B_{R_{\varepsilon}}\left(0\right)}dx\Pi_{\left[0,2\varepsilon\right]}^{P}\left(x,x\right),
\end{split}
\]
for some $R_{\varepsilon}>0$, using $\left(\Pi_{\left[0,2\varepsilon\right]}^{P}\right)^{2}=\Pi_{\left[0,2\varepsilon\right]}^{P}$
and Cauchy-Schwartz. This gives 
\[
\int_{B_{R_{\varepsilon}}\left(0\right)}dx\,e^{-tP}\left(x,x\right)\geq\frac{e^{-2\varepsilon t}}{2}\,,\quad t>1.
\]
Plugging this last inequality into \prettyref{eq:breakup est.} gives
\[
\frac{\int_{B_{R_{\varepsilon}}\left(0\right)}dx\left[Pe^{-tP}\right]\left(x,x\right)}{\int_{B_{R_{\varepsilon}}\left(0\right)}dx\,e^{-tP}\left(x,x\right)}\leq4\varepsilon+ce^{-\varepsilon t}R_{\varepsilon}^{n-1}
\]
from which the theorem follows on choosing $t$ large. 
\end{proof}
\bibliographystyle{siam}
\bibliography{biblio}

\end{document}